\documentclass[11pt]{amsart}
\usepackage{amsmath,amssymb,latexsym,dsfont, mathabx}
\usepackage[left=2.6cm,right=2.6cm,top=2.7cm,bottom=2.7cm]{geometry}
\usepackage{hyperref}
\usepackage{amscd} 
\usepackage{amsfonts,esint}
\usepackage{indentfirst}
\usepackage{verbatim}
\usepackage{amsmath}
\usepackage{amsthm}
\usepackage{enumerate}
\usepackage{graphicx}
\usepackage{subfig}
\usepackage{color}
\usepackage[OT1]{fontenc}
\usepackage[latin1]{inputenc}
\usepackage[english]{babel}
\usepackage{amssymb}
\usepackage{dsfont}
\usepackage{cleverref}

\makeatletter
\DeclareRobustCommand\widecheck[1]{
	{\mathpalette\@widecheck{#1}}}
\def\@widecheck#1#2{
		\setbox\z@\hbox{\m@th$#1#2$}
	\setbox\tw@\hbox{\m@th$#1%
		\widehat{%
			\vrule\@width\z@\@height\ht\z@
			\vrule\@height\z@\@width\wd\z@}$}%
	\dp\tw@-\ht\z@
	\@tempdima\ht\z@ \advance\@tempdima2\ht\tw@ \divide\@tempdima\thr@@
	\setbox\tw@\hbox{%
		\raise\@tempdima\hbox{\scalebox{1}[-1]{\lower\@tempdima\box
				\tw@}}}%
	{\ooalign{\box\tw@ \cr \box\z@}}}
\makeatother
\newtheorem{theorem}{Theorem}
\newtheorem{lemma}{Lemma}

\newtheorem{proposition}{Proposition}
\newtheorem{corollary}{Corollary}
\newtheorem{remark}{Remark}

\newtheorem{definition}{Definition}
\numberwithin{equation}{section}

\newcommand{\hmu}{\hat \mu}

\usepackage{dsfont}

\newcommand{\cF}{{\mathcal F}}

\newcommand{\hSigma}{\hat \Sigma}

\newcommand{\supp}{\operatorname{supp}}
\newcommand{\dist}{\operatorname{dist}}

\newcommand{\dive}{\operatorname{div}}

\newcommand{\tr}{ {\mbox{trace}}}

\newcommand{\eps}{\varepsilon}

\newcommand{\cP}{{\mathcal P}}

\newcommand{\mC}{\mathbb{C}}

\newcommand{\mN}{\mathbb{N}}

\newcommand{\mR}{\mathbb{R}}

\newcommand{\mZ}{\mathbb{Z}}

\newcommand{\hv}{\hat v}

\newcommand{\hcT}{\hat {\mathcal T}}

\newcommand{\tf}{\widetilde f}
\newcommand{\tF}{\widetilde F}
\newcommand{\tg}{\widetilde g}
\newcommand{\tG}{\widetilde G}

\newcommand{\bT}{{\bf T}}

\newcommand{\cS}{{\mathcal S}}

\newcommand{\cK}{{\mathcal K}}
\newcommand{\bK}{{\bf K}}

\newcommand{\cL}{{\mathcal L}}

\newcommand{\tlambda}{\widetilde \lambda}

\title[Weyl law of transmission eigenvalues and the completeness]{The Weyl law of transmission eigenvalues and the completeness  of generalized transmission eigenfunctions}

\author{Hoai-Minh Nguyen}
\address[Hoai-Minh Nguyen]{Ecole Polytechnique F\'ed\'erale de Lausanne, EPFL,  CAMA, 
\newline\indent Station 8,  CH-1015 Lausanne, Switzerland.}
\email{hoai-minh.nguyen@epfl.ch}

\author{Quoc-Hung Nguyen}
\address[Quoc-Hung Nguyen]{Institute of Mathematical Sciences, ShanghaiTech University, 
	\newline\indent 393 Middle Huaxia Road, Pudong,  Shanghai, 201210, China.}
\email{qhnguyen@shanghaitech.edu.cn}

\begin{document}

\maketitle 

\begin{abstract} The transmission problem is a system of two second-order elliptic equations of  two unknowns  equipped with the Cauchy data on the boundary. After four decades of research motivated by scattering theory, the spectral properties of this problem are now known to depend on a type of contrast between coefficients near the boundary.  
Previously, we established the discreteness of eigenvalues for a large class of anisotropic coefficients which is related to the celebrated complementing conditions due to Agmon, Douglis, and Nirenberg. In this work, we establish the Weyl law for the eigenvalues  and the completeness of the generalized eigenfunctions for this class of coefficients under an additional mild assumption on the continuity of the coefficients.  The analysis is new and based on the $L^p$ regularity  theory for the transmission problem established here. It also involves a subtle application of the spectral theory for the Hilbert Schmidt operators.  Our work extends largely known results in the literature which are mainly devoted to  the isotropic case with $C^\infty$-coefficients. 
\end{abstract}


\section{Introduction}

The  transmission eigenvalue problem plays a role in the   inverse scattering theory for inhomogeneous media.  This eigenvalue  problem is connected to the injectivity of the relative scattering operator \cite{CM88}, \cite{Kirsch86}. Transmission eigenvalues are  related to interrogating frequencies for which there is an incident field that is  not scattered by the medium.  In the acoustic setting,  the  transmission problem is a system of two second-order elliptic equations  of two unknowns  equipped with the Cauchy data on the boundary.   After four decades of extensive study, the spectral properties are known to depend on a type of contrasts of the media near the boundary (i.e., a difference of some relation of the respective coefficients in each of the equations)."
Natural and interesting questions on  the inverse scattering theory include:  {\it discreteness} of the spectrum (see e.g. \cite{CCG10, BCH11, Sylvester12, LV12, MinhHung1})  {\it location} of transmission eigenvalues (see  \cite{CGH10, CL12, Vodev15, Vodev18},  and also  \cite{CMV20} for the application in time domain), and the {\it Weyl law} of transmission eigenvalues and the {\it completeness} of the generalized eigenfunctions (see e.g. \cite{LV12, LV12-A, BP13, LV15, Robbiano16}).  We refer the reader to \cite{CCH16} for a recent, and self-contained introduction  to the transmission problem and its applications.

This paper concerns the Weyl law of  eigenvalues and the completeness of the generalized eigenfunctions of the transmission problem  in the time-harmonic acoustic setting.  Let us  introduce its  mathematical formulation. Let  $\Omega$ be a bounded, simply connected, open subset of $\mR^d$ of class $C^2$ with $d\geq 2$. Let $A_1,A_2$ be two real,  symmetric matrix-valued functions,  and let  $\Sigma_1,\Sigma_2$ be two bounded  positive functions that are all defined in $\Omega$. Assume that $A_1$ and $A_2$ are uniformly elliptic,  and $\Sigma_1$ and  $\Sigma_2$ are bounded below by a positive constant in $\Omega$, i.e., for some constant $\Lambda \ge 1$,  one has,  for $j=1, 2$, 
\begin{equation}\label{condi1}
\Lambda^{-1} |\xi|^2\leq \langle  A_j(x) \xi, \xi \rangle \leq \Lambda |\xi|^2  \quad \mbox{ for all } \xi \in \mR^d, \mbox{ for a.e. }  x\in \Omega, 
\end{equation}
and
\begin{equation}\label{condi2}
\Lambda^{-1}  \leq \Sigma_j(x)\leq \Lambda \mbox{ for a.e. } x \in \Omega.  
\end{equation}
Here and in what follows,   $\langle \cdot, \cdot \rangle$ denotes the Euclidean scalar product in $\mR^d$ and $|\cdot|$ is the corresponding norm.  A complex number $\lambda$ is called an eigenvalue of the transmission eigenvalue problem associated with the pairs $(A_1, \Sigma_1)$ and $(A_2, \Sigma_2)$ in $\Omega$ if there is a non-zero pair of functions $(u_1, u_2) \in [H^1(\Omega)]^2$ that satisfy the system
\begin{equation}\label{pro1}  
\left\{\begin{array}{cl}
\dive(A_1 \nabla u_1) - \lambda\Sigma_1 u_1= 0 &\text{ in}~\Omega, \\[6pt]
\dive(A_2 \nabla u_2) - \lambda\Sigma_2 u_2= 0 &\text{ in}~\Omega, \\[6pt]
u_1 =u_2, \quad  A_1 \nabla u_1\cdot \nu = A_2 \nabla u_2\cdot \nu  & \text{ on }\Gamma. 
\end{array}\right. 
\end{equation}
Here and in what follows,  $\Gamma$ denotes $\partial \Omega$, and  $\nu$ denotes the outward, normal, unit vector on $\Gamma$. Such a pair  $(u_1, u_2)$ is then called an eigenfunction pair of \eqref{pro1}. 

The Weyl law of transmission eigenvalues has been investigated under various assumptions on $(A_1, \Sigma_1)$ and $(A_2, \Sigma_2)$. Robbiano \cite{Robbiano16} (see also \cite{Robbiano13}) gives the sharp order of the counting number when $A_1 = A_2 = I$, and  $\Sigma_2 \neq \Sigma_1 = 1$ near the boundary and $\Sigma_2$ is smooth.  The analysis is based on both the microanalysis (see e.g. \cite{GG94, Zworski12}) and the regularity theory for the transmission problem.   
In \cite{LV12-A}, Lakshtanov and Vainberg obtained similar results when $A_1 = I$, $\Sigma_1 = 1$, under certain assumptions on $A_2$ and $\Sigma_2$. In particular, they required that $\Sigma_2^{-1}A_2 - I$ is positive definite or negative definite in the whole domain $\Omega$. They also investigated the order of the counting functions for positive and negative eigenvalues under different assumptions on $A_2$ and $\Sigma_2$ (see also \cite{PS14, LV15}) via concepts on billiard trajectories.  In the isotropic case, the Weyl law for the remainder was established by Petkov and Vodev \cite{PV17} and Vodev \cite{Vodev18, Vodev18-2, Vodev19} for $C^\infty$ coefficients that satisfy the conditions  \eqref{cond1} and \eqref{cond2}. The case where $A_1 = A_2 $ and represent scalar functions was also investigated in their work. Their analysis  is heavily based on  microanalysis and required a strong smoothness condition. In addition, their work involved a delicate analysis on the Dirichlet to Neumann maps using non-standard parametrix construction  initiated by Vodev \cite{Vodev15}.  It is not clear how one can improve the $C^\infty$ condition  and  extend their results to the anisotropic setting using their analysis.  Concerning the completeness of the generalized eigenfunctions, we want to mention the work of Robbiano \cite{Robbiano13} where  $A_1 = A_2 = I$ and  $\Sigma_2 \neq \Sigma_1 = 1$, and the work of Bl\"asten and P\"{a}iv\"{a}rinta \cite{BP13} where  $A_1 = A_2 = I$, and  $\Sigma_2  - \Sigma_1 = \Sigma_2 - 1 > 0$ and smooth in $\bar \Omega$.

In this paper, we investigate the Weyl law of eigenvalues and the completeness of the generalized eigenfunctions for transmission problem under quite general assumptions on $A_1, \, A_2, \,  \Sigma_1,\,  \Sigma_2$. These  are only imposed on the boundary of $\partial \Omega$ except for  the continuity requirement. The starting point and one of the main motivations of our work are our discreteness result established in \cite{MinhHung1}. We demonstrated the discreteness holds if $A_1, \,  A_2,  \, \Sigma_1, \, \Sigma_2$ are continuous in a neighborhood of the boundary $\Gamma$, and satisfy the following two conditions, with $\nu = \nu(x)$:  
 \begin{equation}\label{cond1}
\langle A_2(x) \nu, \nu\rangle \langle A_2(x) \xi, \xi \rangle  - \langle A_2(x) \nu,  \xi \rangle^2 
\neq \langle A_1(x) \nu, \nu\rangle \langle A_1(x) \xi, \xi \rangle  - \langle A_1(x) \nu,  \xi \rangle^2, 
\end{equation}
for all $x \in \Gamma$ and for all $\xi \in \mR^d \setminus \{0 \}$ with  $\langle \xi, \nu \rangle =0$,  and 
\begin{equation}\label{cond2}
\big\langle  A_2(x) \nu, \nu \big\rangle \Sigma_2(x) \neq \big\langle  A_1(x) \nu, \nu \big\rangle \Sigma_1(x),   \; \forall \, x \in \Gamma. 
\end{equation}
Condition \eqref{cond1} is equivalent to the celebrated complementing condition due to Agmon, Douglis, and Nirenberg \cite{ADNII} (see also \cite{ADNI}). The explicit formula given here was derived in \cite{Ng-WP}.

In this paper, we establish that if  conditions \eqref{cond1} and \eqref{cond2} hold  then  the Weyl law for transmission eigenvalues and the completeness of the generalized eigenfunctions hold  as well,  under the mild assumption that the coefficients are continuous in $\bar \Omega$. More precisely, we have 

\begin{theorem}\label{thm1} Assume that $A_1, \, A_2, \, \Sigma_1, \, \Sigma_2\in C^0(\bar \Omega)$,  and \eqref{cond1} and \eqref{cond2} hold.  Then 
	\begin{align}\label{thm1-WL1}
	N(t): = \#\Big\{k \in \mN: |\lambda_k|\leq t \Big\}= \mathbf{c} t^{\frac{d}{2}}+o(t^{\frac{d}{2}}) \text{ as}~~t \to  + \infty, 
	\end{align}
	where 
	\begin{equation}\label{thm1-c}
		\mathbf{c}=\frac{1}{(2\pi)^d} \sum_{j=1}^2 \int_{\Omega} \Big|\Big\{\xi \in \mR^d: \langle  A_j(x) \xi, \xi \rangle<\Sigma_j(x)\Big\} \Big| \, dx. 
	\end{equation}
\end{theorem}

For a measurable subset $D$ of $\mR^d$, we denote $|D|$ its  (Lebesgue) measure. 

\medskip 

We also have 

\begin{theorem} \label{thm1-C} Assume that $A_1, \, A_2, \, \Sigma_1, \, \Sigma_2 \in C^0(\bar \Omega)$,  and \eqref{cond1} and \eqref{cond2} hold.    Then the generalized eigenfunctions are complete in $[L^2(\Omega)]^2$. 
\end{theorem}

\begin{remark} \rm As a direct consequence of either \Cref{thm1} or \Cref{thm1-C}, the number of eigenvalues of the transmission problem is infinite. As far as we know, this fact is new under the general assumptions stated here. 
\end{remark}

Some comments on \Cref{thm1} and \Cref{thm1-C}  are in order. In the conclusion of  \Cref{thm1}, the multiplicity of eigenvalues is taken into account. The meaning of the multiplicity is understood as follows. One can show (see \cite{MinhHung1}, and also \Cref{thm2}) that the well-posedness of the following system in $[H^1(\Omega)]^2$:  
\begin{equation}\label{pro1az*} 
\left\{ \begin{array}{cl}
\dive(A_1 \nabla u_1) - \lambda\Sigma_1 u_1=  \Sigma_1 f_1 &\text{ in}~\Omega, \\[6pt]
\dive(A_2 \nabla u_2) - \lambda\Sigma_2 u_2= \Sigma_2 f_2 &\text{ in}~\Omega, \\[6pt]
u_1 =u_2, \quad  A_1 \nabla u_1\cdot \nu = A_2 \nabla u_2\cdot \nu & \text{ on }\Gamma, 
\end{array}
\right.
\end{equation}
holds for all $(f_1,f_2) \in [L^2(\Omega)]^2$ and for some $\lambda \in \mC$ under the assumptions of \Cref{thm1}. We then define the operator $T_\lambda: [L^2(\Omega)]^2 \to [L^2(\Omega)]^2$ by 
\begin{equation}\label{def-T-*}
T_\lambda(f_1, f_2) = (u_1, u_2)  \mbox{ where $(u_1, u_2)$ is the unique solution of \eqref{pro1az*}}. 
\end{equation}
We can also prove that such a $T_\lambda$ is compact using a priori estimates.  If $\lambda_j$ is an eigenvalue of the transmission problem,  then $\lambda_j \neq \lambda$, and  $\lambda_j - \lambda$ is a characteristic value  of $T_\lambda$ (i.e., $(\lambda_j - \lambda)^{-1}$ is its eigenvalue) and conversely. One can show that the multiplicity of the characteristic values $\lambda_j - \lambda$
and $\lambda_j - \hat \lambda$
(which are the multiplicity of $(\lambda_j - \lambda)^{-1}$ and $(\lambda_j - \hat \lambda)^{-1}$, see \Cref{def-1} below) associated with $T_\lambda$ and $T_{\hat \lambda}$ are the same as long as $T_\lambda$ and $T_{\hat \lambda}$ are well-defined (see \Cref{rem-multiplicity}). Hence, the multiplicity of eigenvalues that are associated with $T_\lambda$ is independent of $\lambda$ and  it is used in Assertion \eqref{thm1-WL1}. 
One can also prove that  $T_\lambda$ and $T_{\hat \lambda}$ have the same set of the generalized eigenfunctions. 
In \Cref{thm1-C}, the generalized eigenfunctions are associated to such a $T_\lambda$. We recall that the generalized eigenfunctions are complete in  $[L^2(\Omega)]^2$ if the subspace spanned by them is dense in  $[L^2(\Omega)]^2$.

\medskip 
Recall that, see e.g.  \cite[Definition 12.5]{Agmon}: 

\begin{definition} \label{def-1} \rm Let $\gamma$ be an eigenvalue of a linear continuous operator $A: H \to H$ where $H$ is a Hilbert space. A non-zero vector $v$ is a generalized eigenvector of $A$ corresponding to $\gamma$ if
$(\gamma I -A)^k v=0$ holds for some positive integer $k$. 
The set of all generalized eigenvectors of $A$ corresponding to the eigenvalue $\gamma$ together with the origin in $H$, forms a subspace of $H$, whose dimension is the multiplicity of  $\gamma$. 
\end{definition}

\Cref{thm1} gives the order of the counting function $N(t)$
and its first-order approximation.  \Cref{thm1} and \Cref{thm1-C} provide new general conditions on the coefficients for which the Weyl law and the completeness of the generalized eigenfunctions hold. These conditions are imposed only on the boundary and the regularity assumption is very mild.  

\begin{remark} \rm It is worth noting that the convention of eigenvalues of the transmission problem in the work of Lakshtanov and Vainberg is similar to ours and different from that of Robbiano (also the work Petkov and Vodev, and Vodev mentioned above) where $\lambda^2$ is used in \eqref{pro1} instead of $\lambda$ (where $\lambda$ is used but $t^2$ is considered instead of $t$ in the formula of the counting function). 
\end{remark}

\begin{remark} \rm In \cite{PV17}, Petkov and Vodev considered the isotropic setting and obtained  a shaper estimate for the remainder of  \eqref{thm1-WL1}  as in the spirit of H\"ormander \cite{Hormander68}. Other refined estimates were given in \cite{Vodev18, Vodev18-2,Vodev19} and  are obtained under the $C^\infty$ smoothness assumption. Under the continuity assumption on the smoothness of coefficients, a better estimate for the remainder of \eqref{thm1-WL1} as in \cite{PV17} is implausible. Nevertheless, it is interesting to obtain better estimates for the remainder as in \cite{PV17, Vodev18, Vodev18-2,Vodev19} for sufficiently regular coefficients and/or for the anisotropic setting. 
\end{remark}

Our strategy of the analysis is to develop the approach in \cite{MinhHung1} at the level where  one can apply the general spectral  theory for Hilbert-Schmidt operators in Hilbert space as given in  Agmon \cite{Agmon} (see also \cite{Agmon62}). Two important steps are follows. One is on  sharp estimates for $\| T_\lambda \|_{L^p \to W^{1, p}}$ for $p>1$ and its consequences (see \Cref{thm2}) for large $|\lambda|$ with an  appropriate direction. This,  in particular,  shows that $\bT : = T_{\hmu_1} \circ  \dots \circ T_{\hmu_{k+1}}$ with $k = [d/2]$ is a Hilbert - Schmidt operator (see \Cref{pro-HS}) for an appropriate choice of $\hmu_j \in \mC$.  The analysis of this part is on the regularity theory of the transmission problems in $L^p$-scale. This is one of the cores of this paper and has its own interest. To this end, we first investigate the corresponding problems in the whole space and in a half space with constant coefficients,  and then use the freezing-coefficient technique. The analysis also involves  the Mikhlin-H\"ormander multiplier theorem (in particular the theory of singular integrals) and Gagliardo-Nirenberg interpolation inequalities.  The second step is to apply the spectral theory for Hilbert-Schmidt operators. 
To this end, we use the estimates for $T_\lambda$ to obtain an approximation of the trace of the kernel of the product of $\bT$ and its appropriate modified operator (see \Cref{pro-trace}). The approximation of the trace of the kernel is then used to derive information for the Weyl law via a formula for eigenvalues established in \Cref{pro-trace0}. This formula is derived from the spectral theory of Hilbert-Schmidt operator and is interesting itself.  The completeness of the generalized eigenfunctions follows directly from the estimates for $T_\lambda$ in \Cref{thm2} where we pay special attention to the possible directions of $\lambda$ where the information can be derived, after applying the spectral theory in \cite{Agmon}.  

\begin{remark} \rm 
We use the regularity theory and spectral theory for Hilbert-Schmidt operators to investigate the Weyl law, which was also presented by Robbiano  \cite{Robbiano13}. Nevertheless, the way we derive the regularity theory in this paper is distinct from \cite{Robbiano13}, which involved Carleman's inequalities and the theory of microanalysis. The way we explore the information of Hilbert-Schmidt operators allows us to exactly obtain the first term of the Weyl Law 
in \eqref{thm1-WL1} instead of its magnitude  order as in \cite{Robbiano13}. 
\end{remark}

We  propose a new approach to establish the Weyl law of  eigenvalues and the completeness of the generalized eigenfunctions. This allows us  to obtain new significant results and strongly  weaken the smoothness assumption in  various known results,  that is out of reach previously.   The transmission problem also appears naturally for electromagnetic waves. In this case, it  is a system of two Maxwell systems   equipped the Cauchy data on the boundary. The spectral theory of the transmission problem for electromagnetic waves is much less known.  On this aspect, we point the reader to \cite{Cakoni-Ng20} on the discreteness,  and to  \cite{HM18} on the completeness. More information can be found in the references therein. The analysis in this paper will be developed for the Maxwell setting  in our forthcoming work.

The transmission problem has an interesting connection with the study of negative-index materials which are modeled by the Helmholtz or Maxwell equations with sign changing coefficients.  
In fact, our work has its roots in \cite{Ng-WP} where the stability of solutions of the Helmholtz equations with sign changing coefficients was studied. Concerning the Maxwell equations, the stability was studied in \cite{NgSil}. It is not coincident that  the transmission problem and the Helmholtz equations with sign-changing coefficients share some common analysis. In fact, using reflections (a class of changes of variables), the Cauchy problems appear naturally in the context of the Helmholtz  with sign-changing coefficients as first observed in \cite{Ng-Complementary} (see also \cite{Ng-Superlensing-Maxwell} for the Maxwell setting). 
Other properties of the Cauchy problems related to resonant (unstable) aspects and applications of negative-index materials such as cloaking and superlensing can be found in  \cite{Ng-Superlensing, Ng-CALR, Ng-Negative-Cloaking, Ng-CALR-O, Ng-Negative-Cloaking-M,Ng-CALR-O-M} and the references therein. 

\medskip

The paper is organized as follows. In \Cref{sect-Notation}, we introduce several notations used throughout the paper. 
In \Cref{sect-Lp}, we establish \Cref{thm2}, which describes  the regularity theory for the transmission problem in $L^p$-scale. 
In \Cref{sect-HS}, we recall some definitions, properties of  Hilbert-Schmidt operators,  and their
finite double-norms. We then derive their applications in the context of the transmission problem. The main result of this section is \Cref{pro-trace0}, which  is derived from \Cref{thm2}. The Weyl law and the completeness are then established in \Cref{sect-WL} and in \Cref{sect-C}, respectively. 
 
\section{Notations}\label{sect-Notation}

We denote, for $\tau > 0$,  
\begin{equation*}
\Omega_\tau=\Big\{x\in \Omega: \dist(x,\Gamma)<\tau \Big\}.
\end{equation*}
For $d \ge 2$, set
$$
\mR^d_+ = \Big\{x \in \mR^d; x_d > 0 \Big\} \quad \mbox{ and } \quad \mR^d_0 = \Big\{x \in \mR^d; x_d = 0 \Big\}.  
$$
We will identify $\mR^d_0$ with $\mR^{d-1}$ in several places. 

For $\theta \in \mR$ and $a>0$, denote 
\begin{equation}
\cL (\theta, a) = \Big\{r e^{i\theta} \in \mC: r \ge a \Big\}.
\end{equation}

\section{Regularity theory for transmission problems}\label{sect-Lp}

In this section, we establish several estimates for $T_\lambda$ for appropriate values of $\lambda$.  
The main results are as follows.

\begin{theorem}\label{thm2} Let $\eps_0>0$ and  $\Lambda \ge 1$.   Assume that \eqref{condi1} and \eqref{condi2} hold, and $A_1,\, A_2, \, \Sigma_1, \, \Sigma_2$ are continuous in $\bar \Omega$.  Assume that \eqref{cond1} and \eqref{cond2} hold in the following sense, with $\nu = \nu(x)$,  
\begin{equation}\label{cond1-thm2}
\left| \langle A_2(x) \nu, \nu\rangle \langle A_2(x) \xi, \xi \rangle  - \langle A_2(x) \nu,  \xi \rangle^2  - \Big(
 \langle A_1(x) \nu, \nu\rangle \langle A_1(x) \xi, \xi \rangle  - \langle A_1(x) \nu,  \xi \rangle^2 \Big) \right| \ge \Lambda^{-1} |\xi|^2, 
\end{equation}
for all $x \in \Gamma$ and for all $\xi \in \mR^d \setminus \{0 \}$ with  $\langle \xi, \nu \rangle =0$,  and 
\begin{equation}\label{cond2-thm2}
\Big| \langle  A_2(x) \nu, \nu \rangle \Sigma_2(x) -  \langle  A_1(x) \nu, \nu \rangle \Sigma_1(x) \Big| \ge \Lambda^{-1},   \; \forall \, x \in \Gamma. 
\end{equation} 
There exist two positive constants $\Lambda_0$ and $C$ depending only on $\Lambda$, $\eps_0$, $\Omega$, and the continuity modulus  of $A_1$, $A_2$, $\Sigma_1$, and $\Sigma_2$ in $\bar \Omega $ such that for $\theta \in \mR$  with $\inf_{n\in \mathbb{Z}}|\theta-n\pi|\geq \eps_0$, and for $\lambda \in \cL(\theta, \Lambda_0)$, the following fact holds: for $g= (g_1,g_2)\in [L^2(\Omega)]^2$,  there exists a unique solution   $u=(u_1,u_2)\in [H^1(\Omega)]^2$ of the system 
	\begin{equation}\label{pro1-sec1'}
		\left\{ \begin{array}{cl}
			\dive (A_1\nabla u_1) -\lambda\Sigma_1 u_1=g_1 ~~&\text{ in}~\Omega,\\[6pt]
			\dive (A_2\nabla u_2) -\lambda \Sigma_2 u_2= g_2~~&\text{ in}~\Omega,\\[6pt]
			u_1-u_2= 0,~~ (A_1\nabla u_1-A_2\nabla u_2) \cdot \nu =0 & \text{ on } \Gamma. 
		\end{array} \right. 
	\end{equation}
Moreover, for $1<p<\infty$, 
\begin{equation}\label{thm2-st1}
 \| \nabla u\|_{L^p(\Omega)} + |\lambda|^{1/2} \| u \|_{L^p(\Omega)}   \le C |\lambda|^{- \frac{1}{2}} \| g \|_{L^p(\Omega)}. 
\end{equation}
As a consequence, we have
\begin{itemize}
\item for $1<p<d$ and  $p\leq q\leq \frac{dp}{d-p}$,
 \begin{equation}\label{thm2-st2}
||u||_{L^{q}(\Omega)}\leq C |\lambda|^{-1+\frac{d}{2}\left(\frac{1}{p}-\frac{1}{q}\right)}\| g \|_{L^p(\Omega)}, 
\end{equation}
\item for $p>d$,  \begin{equation}\label{thm2-st3}
\|u\|_{L^{\infty}(\Omega)}\leq C |\lambda|^{-1+\frac{d}{2p}}\| g \|_{L^p(\Omega)},
\end{equation}
\item  for $p> d$ and $q = \frac{p}{p-1}$, 
\begin{equation}\label{thm2-st4}
\|u\|_{L^{q}(\Omega)}\leq C |\lambda|^{-1+\frac{d}{2}-\frac{d}{2q}}\| g \|_{L^1(\Omega)}.
\end{equation}
\end{itemize}
\end{theorem}

The remainder of this section contains two subsections, which are organized as follows.  In the first subsection, we establish several lemmas used in the proof of \Cref{thm2}. The proof of \Cref{thm2} is given in the second subsection. 

\subsection{Preliminaries}
In this section, we establish several results used in the proof of \Cref{thm2},  which is based on freezing coefficient technique. We begin with the corresponding settings/variants  with constant coefficients in $\mR^d$ and in $\mR^{d}_+$. The first one is 

\begin{lemma} \label{lem1} Let $d \ge 2$, $\Lambda \ge 1$, $\eps_0 > 0$, $1<p<\infty$,  and let  $A$ and  $\Sigma$ 
be a symmetric matrix and a non-zero real constant, respectively. Assume that $\inf_{n \in \mZ} |\theta - n \pi| \ge \eps_0$ and  $\lambda \in \cL(\theta, 1)$, 
\begin{equation}\label{lem1-ellipticiy}
\Lambda^{-1} \le A \le \Lambda \quad \mbox{ and } \quad \Lambda^{-1} \le |\Sigma| \le \Lambda. 
\end{equation}
For $g \in L^p(\mR^d)$ and $G \in [L^p(\mR^d)]^d$, let $u \in W^{1, p} (\mR^d)$ be the unique solution of 
$$
\dive (A \nabla u) - \lambda \Sigma u = g + \dive (G) \mbox{ in } \mR^d.  
$$
We have
\begin{equation}\label{lem1-st1-1}
 |\lambda|^{1/2} \| \nabla u \|_{L^p(\mR^d)} + |\lambda| \| u \|_{L^p(\mR^d)} \le C \Big( \| g\|_{L^p(\mR^d)} + |\lambda|^{1/2} \| G \|_{L^p(\mR^d)} \Big), 
\end{equation}
and, if $G = 0$,  
\begin{equation}\label{lem1-st1-2}
 \| \nabla^2 u \|_{L^p(\mR^d)}  \le C  \| g\|_{L^p(\mR^d)}. 
\end{equation}
Here $C$ denotes a positive constant depending only on $p$, $d$, $\Lambda$, and $\eps_0$. 
\end{lemma}

Here and in what follows, for two $d \times d$ symmetric matrices $M_1$ and $M_2$, we denote $M_1 \ge M_2$ (resp. $M_1 \le M_2$) if $\langle M_1 \xi, \xi \rangle \ge \langle M_2 \xi, \xi \rangle$ (resp. $\langle M_1 \xi, \xi \rangle \le \langle M_2 \xi, \xi \rangle$)  for all $\xi \in \mR^d$. 

\begin{proof} For an appropriate function/vector field  $f$ defined in $\mR^d$, let $\cF f$ denote its Fourier transform.  We have 
\begin{equation*}
\cF u(\xi) = - \frac{\cF g (\xi) + i \xi \cdot \cF G (\xi)}{ \langle A \xi, \xi \rangle + \lambda \Sigma }. 
\end{equation*}
Set 
$$
m(\xi) = \frac{1}{\langle A \xi, \xi \rangle + \lambda \Sigma}. 
$$
One can check that 
$$
|\xi|^\ell |\nabla^\ell m(\xi)| \le C_\ell |\lambda|^{-1} \mbox{ for } \ell \in \mathbb{N}. 
$$
It follows from Mikhlin-H\"ormander's multiplier theorem, see e.g. \cite[Theorem 5.2.7]{Grafakos08}, that 
\begin{equation*}
\| u \|_{L^p(\mR^d)} \le C |\lambda|^{-1} \| g\|_{L^p(\mR^d)}. 
\end{equation*}
The other estimates in Assertion~\eqref{lem1-st1-1} and \eqref{lem1-st1-2}  can be derived in the same manner. The proof is complete. 
\end{proof}

Here is a result on a half space. 

\begin{lemma}\label{lem2} Let $A_1,A_2$ be two  constant, symmetric  matrices,  and let $\Sigma_1,\Sigma_2$ be two non-zero, real constants. 
Assume that, for some $\Lambda \ge 1$,  			 
\begin{equation}\label{cond-0-sect1}
\Lambda^{-1} \le A_1, A_2 \le \Lambda, \quad \Lambda^{-1} \le |\Sigma_1|, |\Sigma_2| \le \Lambda, 
\end{equation}
	\begin{equation}\label{cond-1-sec1}
	\Big|\langle A_2 e_d, e_d \rangle \langle A_2 \xi, \xi \rangle  - \langle A_2 e_d,  \xi \rangle^2  - \langle A_1 e_d, e_d \rangle \langle A_1 \xi, \xi \rangle  - \langle A_1 e_d, \xi \rangle^2 \Big| \ge \Lambda^{-1} |\xi|^2 \quad \forall \, \xi \in \cP, 
	\end{equation}
	where $
\cP=	\big\{\xi \in \mR^d; \langle \xi, e_d \rangle = 0  \big\}, $
	and 
	\begin{equation}\label{cond-2-sec1}	
	\Big|\big\langle  A_2 e_d, e_d \big\rangle \Sigma_2  -  \big\langle  A_1 e_d, e_d \big\rangle \Sigma_1 \Big| \ge \Lambda^{-1}.
	\end{equation}
	 Let $p > 1$,  $\eps_0 > 0$, $g_1, g_2   \in L^p(\mR^d_+)$,  $G_1, G_2   \in [L^p(\mR^d_+)]^d$, and $\varphi \in W^{1-1/p, p}(\mR^d_0)$. 
	 There exist two positive constants $C$ and $\Lambda_0$ depending only on $\Lambda$ and $\eps_0$ such that for $\theta \in \mR$ with $\min_{n \in \mZ} |\theta - n \pi| \ge \eps_0$ and for $\lambda \in \cL(\theta, \Lambda_0)$, there exists a unique solution   $ u = (u_1,u_2) \in [W^{1, p}(\mR^d_+)]^2$ of  the system
	\begin{equation}\label{pro3}
	\left\{ \begin{array}{cl}
	\dive(A_1 \nabla v_1) -\lambda \Sigma_1 v_1=g_1+ \dive(G_1)~~&\text{ in}~\mathbb{R}_+^d,\\[6pt]
	\dive(A_2 \nabla v_2) -\lambda  \Sigma_2 v_2= g_2+ \dive(G_2)~~&\text{ in}~\mathbb{R}_+^d,\\[6pt]
	v_1-v_2= \varphi,~~ (A_1 \nabla v_1 - G_1)\cdot e_d - (A_2 \nabla v_2 - G_2)\cdot e_d = 0 & \text{ on } \mR^d_0.
	\end{array} \right. 
	\end{equation} 
	Moreover,
	\begin{multline}
	\label{lem2-st1}
	\| \nabla v \|_{L^{p}(\mathbb{R}_+^d)}+ |\lambda|^{1/2}	\| v \|_{L^{p}(\mathbb{R}_+^d)}
	\leq C\left(  |\lambda|^{-1/2}\|g\|_{L^{p}(\mathbb{R}_+^d)}+\|G\|_{L^{p}(\mathbb{R}_+^d)}  \right.  \\[6pt]+ \left.
\lambda^{1/2 - 1/ (2p)} \| \varphi \|_{L^p(\mR^d_0)} + 	
	 \|\varphi \|_{\dot W^{1-1/p, p}(\mR^d_0)}  \right). 
	\end{multline}
\end{lemma}

\begin{proof} We only establish \eqref{lem2-st1}. The uniqueness for \eqref{pro3} is a consequence of \eqref{lem2-st1}. The existence of $(v_1, v_2)$ follows from the proof of \eqref{lem2-st1} and is omitted.   Let $u_j \in W^{1, p}(\mR^d)$ be the unique solution of the equation 
$$
\dive (A_j \nabla u_j) - \lambda \Sigma_j u_j = g_j \mathds{1}_{\mR^d_+} + \dive (G_j \mathds{1}_{\mR^d_+}) \mbox{ in } \mR^d. 
$$
It follows from  \Cref{lem1} that 
$$
\| \nabla u_j \|_{L^p(\mR^d)} + |\lambda|^{1/2} \| u_j \|_{L^p(\mR^d)} \le C \Big( |\lambda|^{-1/2} \|  g_j\|_{L^p (\mR^d_+) } +  \|G_j \|_{L^p(\mR^d_+)}  \Big). 
$$
We have
$$
|\lambda|^{1/2 - 1/(2p)} \| u_j\|_{L^p(\mR^d_0)} \le C \Big(  \| \nabla u_j \|_{L^p(\mR^d_+)} + |\lambda|^{1/2} \| u_j \|_{L^p(\mR^d_+)}   \Big), 
$$ 
$$
(A_j \nabla  u_j - G_j) \cdot e_d = 0 \mbox{ on } \mR^d_0, 
$$
and, by the trace theory,  
$$
\| u_j\|_{W^{1 - 1/p, p}(\mR^d_0)} \le C \| u_j \|_{W^{1, p}(\mR^d_+)}. 
$$ 
Therefore, without loss of generality,  one might assume that  $g_1=g_2 = 0$ and $G_1=G_2 = 0$. This will be assumed from now on. 

Let $\hat v_j(\xi', t)$ for $j =1, 2$ and $\hat \varphi(\xi', t)$ be the Fourier transform of $v_j$ and $\varphi$ with respect to $x' \in \mR^{d-1}$, i.e., for $(\xi', t) \in \mR^{d-1} \times (0, + \infty)$, 
\begin{equation*}
\hat v_j(\xi', t) = \int_{\mR^{d-1}} v_j(x', t) e^{-i x' \cdot \xi' }\, dx' \quad  \mbox{ for } j=1, 2, \quad \mbox{ and } \quad  
\hat \varphi(\xi', t) = \int_{\mR^{d-1}} \varphi(x') e^{-i x' \cdot \xi' }\, dx'. 
\end{equation*}
Since 
\begin{equation*}
\dive(A_j \nabla v_j ) -  \lambda \Sigma_j v_j = 0 \mbox{ in } \mR^{d}_+,   
\end{equation*}
it follows that 
\begin{equation*}
a_j v_j''(t) + 2 i b_j v_j'(t) - (c_j + \lambda \Sigma_j) v_j(t) = 0 \mbox{ for } t > 0,  
\end{equation*}
where 
\begin{equation*}
a_j  = (A_j)_{d,d}, \quad b_j =  \sum_{k=1}^{d-1} (A_j)_{d, k} \xi_k, \quad \mbox{ and } \quad c_j = \sum_{k=1}^{d-1} \sum_{l =1}^{d-1} (A_j)_{k, l} \xi_k \xi_l.  
\end{equation*}
Here $(A_j)_{k, l}$ denotes the $(k, l)$ component of $A_j$ for $j=1, 2$ and the symmetry of $A_j$ is used. 
Define, for $j=1, 2$,  
\begin{equation}\label{def-Deltaj-v4}
\Delta_j = - b_j^2 +  a_j (c_j + \lambda \Sigma_j). 
\end{equation}
Denote $\xi = (\xi', 0)$.  Since $A_j$ is symmetric and positive, it is clear that, for $j=1, 2$,  
\begin{equation}\label{est1-C-v4-0}
a_j = \langle A_j e_d, e_d \rangle, \; \;  b_j = \langle A_j \xi, e_d \rangle,    \; \; c_j = \langle A_j \xi, \xi \rangle, \; \;   \mbox{ and }  \; \;  a_j c_j - b_j^2  > 0.  
\end{equation}
Since $\hat v_j(\xi', t) \in L^2(\mR^d_+)$, we have 
\begin{equation*}
\quad \hat v_j(\xi', t) = \alpha_j (\xi') e^{\eta_j (\xi') t}, 
\end{equation*}
for some $\alpha_j(\xi') \in \mC$,  where
\begin{equation*}
\eta_j  =( - i b_j - \sqrt{\Delta_j})/ a_j. 
\end{equation*}
Here $\sqrt{\Delta_j}$ denotes the square root of $\Delta_j$ with positive real part. 
Using the fact that $v_1 - v_2=\varphi $ and $A_1 \nabla v_1 \cdot e_d - A_2 \nabla v_2 \cdot e_d = 0$ on $\mR^d_0$, we derive that 
\begin{equation}\label{equation-C-v4}
\alpha_1  (\xi' ) - \alpha_2 (\xi') = \hat \varphi(\xi')  \quad \mbox{ and } \quad \alpha_1 (\xi') \langle i A_1 \xi + \eta_1 A_1 e_d, e_d \rangle  - \alpha_2 (\xi') 
 \langle i A_2 \xi + \eta_2 A_2 e_d, e_d \rangle  = 0. 
\end{equation}
Note that, by \eqref{est1-C-v4-0}, 
$$
\langle A_j \xi, e_d \rangle - \langle A_j e_d, e_d \rangle  b_j/ a_j = 0. 
$$
The last identity of \eqref{equation-C-v4} is equivalent to
\begin{equation*}
\alpha_1(\xi') \sqrt{\Delta_1} = \alpha_2(\xi') \sqrt{\Delta_2}. 
\end{equation*}
Combining this identity and the first one of \eqref{equation-C-v4} yields 
\begin{equation}\label{def-alphaj-v4}
\alpha_1(\xi') = \frac{\hat \varphi(\xi') \sqrt{\Delta_2}}{\sqrt{\Delta_2} - \sqrt{\Delta_1}}. 
\end{equation}

Extend $v_1(x', t)$ by 0 for  $t < 0$. 
We then obtain 
$$
\cF v_1(\xi) =  - \hat \varphi (\xi') \frac{\sqrt{\Delta_2}}{\sqrt{\Delta_2} - \sqrt{\Delta_1}} \frac{1}{\eta_j - i \xi_d}. 
$$
Here, $\cF$ is the Fourier transform in $\mathbb{R}^d$. 
Set 
$$
g(t) = e^{-|\lambda|^{1/2} t} \mathds{1}_{t \ge 0}  \mbox{ for } t \in \mR \quad \mbox{ and } \quad   \Phi (x) = \varphi(x') g(x_d) \mbox{ for } x \in \mR^d. 
$$
It follows that 
$$
\cF v_1 (\xi) = \cF \Phi (\xi) \frac{\sqrt{\Delta_2}}{\sqrt{\Delta_2} - \sqrt{\Delta_1}} \frac{|\lambda| ^{1/2}+ i \xi_d}{-\eta_j + i \xi_d}.
$$
We have 
\begin{equation*}
|\Delta_2 - \Delta_1|^2 \ge C (|\xi'|^4 + |\lambda|^2),  \quad 
|\Delta_j| \le C (|\xi'|^2 +|\lambda|), 
\end{equation*}
and  
\begin{equation*}
|\Re (\eta_j)| \ge C(|\xi'| + |\lambda|^{1/2}). 
\end{equation*}
As in the proof of \Cref{lem1}, by Mikhlin-H\"ormander's multiplier theorem, see e.g. \cite[Theorem 5.2.7]{Grafakos08}, one has 
\begin{equation}\label{lem2-p1}
\| v_1\|_{L^p(\mR^d)} \le C \| \Phi\|_{L^p(\mR^d)} \le C |\lambda|^{-1/(2p)} \| \varphi\|_{L^p(\mR^{d-1})}. 
\end{equation}

We next deal with $\nabla v_1$. We have 
\begin{equation}
\partial_{t} \hv_1 (\xi', t) = \eta_1 \hv_1(\xi', t) \mbox{ in } \mR^d_+. 
\end{equation}
It is clear that 
$$
\eta_1  = \eta_{1, 1} + \eta_{1,2}, 
$$
where 
$$
\eta_{1, 1} =  -  \frac{\sqrt{a_1 \lambda \Sigma_1}}{a_1}, \quad \mbox{ and } \quad \eta_{1, 2} =  - \frac{i b_1}{a_1}     - \frac{\sqrt{\Delta_1} - \sqrt{a_1 \lambda \Sigma_1}}{a_1}. 
$$
As above, one can prove that 
\begin{equation}\label{lem2-p2-1}
\| \hat \cF^{-1}(\eta_{1, 1}  \hv_1  )\|_{L^p (\mR^d_+)} \le C |\lambda|^{1/2} \| \Phi\|_{L^p(\mR^d)} \le C |\lambda|^{1/2-1/(2p)} \| \varphi\|_{L^p(\mR^{d-1})},  
\end{equation}
and, for some $\gamma > 0$, 
$$
\| \hat \cF^{-1}(\eta_{1, 2}   \hv_1  )\|_{L^p (\mR^d_+)} \le C \| g \|_{L^p(\mR^d_+)}, 
$$
where $\hat \cF^{-1}$ denotes the Fourier inverse with respect to $ \xi' $ in $\mathbb{R}^{d-1}$ and 
$$
\hat g(\xi', t) = i \xi' \hat \varphi(\xi') e^{ - \gamma |\xi'| t}. 
$$
It is clear that $g(x) = \nabla_{x'} v (x)$,  where $v$ is the unique solution of the system  
$$
\Delta_{x'} v + \gamma \partial^2_{x_d} v = 0 \mbox{ in } \mR^d_+ \quad \mbox{ and } \quad v = \varphi \mbox{ on } \mR^d_0
$$
for $\gamma>0$. 
It follows that, see e.g. \cite[Theorem 3.3]{ADNI}, we have 
\begin{equation}\label{lem2-p2-2}
\| g \|_{L^p(\mR^d_+)} \le C \| \varphi\|_{\dot W^{1/p-1, p}(\mR^{d-1})}. 
\end{equation}
Combining \eqref{lem2-p2-1} and \eqref{lem2-p2-2} yields 
\begin{equation}\label{lem2-p2}
\| \partial_{x_d} v_1 \|_{L^p(\mR^d_+)} \le C \| \varphi\|_{\dot W^{1/p-1, p}(\mR^{d-1})}  + C |\lambda|^{1/2-1/(2p)} \| \varphi\|_{L^p(\mR^{d-1})}. 
\end{equation}
By the same manner, we also obtain 
\begin{equation}\label{lem2-p3}
\| \nabla_{x'} v_1 \|_{L^p(\mR^d_+)} \le C \| \varphi\|_{\dot W^{1/p-1, p}(\mR^{d-1})}. 
\end{equation}
From  \eqref{lem2-p1}, \eqref{lem2-p2}, and \eqref{lem2-p3}, we obtain 
\begin{equation}\label{lem2-p4}
|\lambda|^{1/2} \| v_1 \|_{L^p(\mR^d_+)} + \| \nabla_{x} v_1 \|_{L^p(\mR^d_+)} \le C \| \varphi\|_{\dot W^{1/p-1, p}(\mR^{d-1})} + C |\lambda|^{1/2-1/(2p)} \| \varphi\|_{L^p(\mR^{d-1})}.
\end{equation}

Similar to \eqref{lem2-p4}, we also get 
\begin{equation}\label{lem2-p5}
|\lambda|^{1/2} \| v_2 \|_{L^p(\mR^d_+)} + \| \nabla_{x} v_2 \|_{L^p(\mR^d_+)} \le C \| \varphi\|_{\dot W^{1/p-1, p}(\mR^{d-1})} + C |\lambda|^{1/2-1/(2p)} \| \varphi\|_{L^p(\mR^{d-1})}.
\end{equation}
The conclusion thus follows from \eqref{lem2-p4} and \eqref{lem2-p5}. The proof is complete. 
\end{proof}

\begin{remark} \rm 
Assertion~\eqref{pro3} was previously established in \cite{MinhHung1} for $p=2$ (see \cite[the proof of Theorem 4]{MinhHung1}). The analysis given here has its root in \cite{Ng-WP, MinhHung1}. Nevertheless, instead of using Parseval's theorem to derive $L^2$-estimates,  the new ingredient involves Mikhlin- H\"ormander's multiplier theory. 
\end{remark}

We now derive consequences of \Cref{lem1,lem2} via the freezing-coefficient  technique. 
As a consequence of \Cref{lem1}, we have 
\begin{corollary} \label{cor1} Let $d \ge 2$, $p>1$,  $\Lambda \ge 1$, $\eps_0 > 0$,  and let $A$ be a symmetric, matrix-valued function,  and let $\Sigma$ be a  real function defined in $\Omega$. Assume that $A$ and $\Sigma$ are continuous in $\bar \Omega$,  
\begin{equation}\label{cor1-AS}
\Lambda^{-1} \le A \le \Lambda \quad \mbox{ and } \quad \Lambda^{-1} \le |\Sigma| \le \Lambda \mbox{ in } \Omega, 
\end{equation}
for some $\Lambda \ge 1$,   $\inf_{n \in \mZ} |\theta - n \pi| \ge \eps_0$, and  $\lambda \in \cL(\theta, 1)$.  For $g \in L^p(\Omega)$ and $G \in [L^p(\Omega)]^d$, let  $u \in W^{1, p}(\Omega)$ be a solution of 
$$
\dive (A \nabla u) - \lambda \Sigma u = g + \dive (G) \mbox{ in } \Omega.  
$$
We have, for  $\tau > 0$, 
\begin{multline}
\| \nabla u \|_{L^p(\Omega \setminus \Omega_\tau)} + |\lambda|^{1/2} \| u \|_{L^p(\Omega \setminus \Omega_\tau)}    \\[6pt]
\le C \Big( |\lambda|^{-1/2} \| g\|_{L^p(\Omega)} +  \| G\|_{L^p(\Omega)} \Big) + C |\lambda|^{- \frac{1}{2}} \Big( \| \nabla u \|_{L^p(\Omega)} +  |\lambda|^{1/2} \| u \|_{L^p(\Omega)}  \Big). 
\end{multline}
Here $C$ denotes a positive constant depending only on $\Lambda$, $p$, $\eps_0$,  $\tau$, $\Omega$, and  the continuity modulus of $A_1$, $A_2$, $\Sigma_1$, and $\Sigma_2$ in $\overline{\Omega}$.
\end{corollary}

\begin{proof} Let $\chi$ be an arbitrary smooth function with support in $\Omega$. Set $v = \chi u $ in $\Omega$. We have 
$$
\dive (A \nabla v) - \lambda \Sigma v = f + \dive F \mbox{ in } \Omega, 
$$
where 
$$
f = \chi g + A \nabla u \nabla \chi - F \cdot \nabla \chi  \quad \mbox{ and } \quad F = \chi F +  u A \nabla \varphi. 
$$
The conclusion follows from \Cref{lem1} by the freezing-coefficient technique and the computations above. 
\end{proof}

Similarly, as a consequence of \Cref{lem2}, we obtain 

\begin{corollary}\label{cor2}
Let $d \ge 2$, $p>1$, $\eps_0>0$, $\tau > 0$,  and  $\Lambda \ge 1$,  and let $A_1, \, A_2$ be two symmetric, matrix-valued functions,  and let $\Sigma_1, \, \Sigma_2$ be two  real functions defined in $\Omega$.  Assume that  $A_1,\, A_2, \, \Sigma_1, \, \Sigma_2$ are  continuous in $\overline{\Omega}_{2\tau}$,  \eqref{cor1-AS} holds, and \eqref{cond1} and \eqref{cond2} are satisfied. There exist two positive constants $\Lambda_0$ and $C$,  depending only on $\Lambda$, $\eps_0$, and the continuity of $A_1$, $A_2$, $\Sigma_1$, and $\Sigma_2$ in $\overline{\Omega}_{2\tau}$ such that for $\theta \in \mR$, for  $\lambda \in \mathcal{L}(\theta, \Lambda_0)$  with $\inf_{n \in \mathbb{Z}}|\theta- n \pi|\geq \eps_0$, and for $g= (g_1,g_2)\in [L^p(\Omega)]^2$,  and $G = (G_1, G_2) \in [L^p(\Omega)]^d \times [L^p(\Omega)]^d $,  let   $u=(u_1,u_2)\in [W^{1, p}(\Omega)]^2$ be a solution of the system 
	\begin{equation}\label{pro1-sec1}
		\left\{ \begin{array}{cl}
			\dive (A_1\nabla u_1) -\lambda\Sigma_1 u_1=g_1 + \dive (G_1) ~~&\text{ in}~\Omega,\\[6pt]
			\dive (A_2\nabla u_2) -\lambda \Sigma_2 u_2= g_2 + \dive (G_2)~~&\text{ in}~\Omega,\\[6pt]
			u_2-u_1= 0,~~ (A_2\nabla u_2-A_1\nabla u_1 - G_2 + G_1) \cdot \nu =0 & \text{ on } \Gamma. 
		\end{array} \right. 
	\end{equation}
Moreover, we have
	\begin{multline}
 \| \nabla v \|_{L^p( \Omega_\tau)} + |\lambda|^{1/2} \| v \|_{L^p(\Omega_\tau)}  \\[6pt]
\le C \Big( |\lambda|^{-1/2} \| g\|_{L^p(\Omega)} +  \| G\|_{L^p(\Omega)} \Big) + C |\lambda|^{- \frac{1}{2}} \Big( \| \nabla v \|_{L^p(\Omega)} + |\lambda|^{1/2} \| v \|_{L^p(\Omega)}  \Big).  
\end{multline}
Here $C$ denotes a positive constant depending only on $\Lambda$, $p$, $\eps_0$,  $\tau$, $\Omega$, and  the continuity modulus of $A_1$, $A_2$, $\Sigma_1$, and $\Sigma_2$ in $\overline{\Omega}_{2\tau}$.
\end{corollary}

\subsection{Proof of \Cref{thm2}} We first assume the well-posedness of \eqref{pro1-sec1'} and establish \eqref{thm2-st1} - \eqref{thm2-st4}. 

It is clear that \eqref{thm2-st1} is a consequence of \Cref{cor1} and \Cref{cor2}. 

We next deal with 
\eqref{thm2-st2} and \eqref{thm2-st3}.  By Gagliardo-Nirenberg's interpolation inequalities \cite{Gagliardo59, Nirenberg59},  if $p>d$ and $u\in W^{1,p}(\Omega)$, then  $u\in C(\overline{\Omega})$ and 
\begin{equation}\label{Z1}
\|u\|_{L^\infty(\Omega)}\leq C\|u\|_{W^{1,p}(\Omega)}^{\frac{d}{p}}\|u\|_{L^p(\Omega)}^{1-\frac{d}{p}},
\end{equation}
and if $1<p<d$, 
and $u\in W^{1,p}(\Omega)$, then,   for $p\leq q<\frac{dp}{d-p}$, 
\begin{equation}\label{Z1'}
\|u\|_{L^q(\Omega)}\leq C\|u\|_{W^{1,p}(\Omega)}^{d(\frac{1}{p}-\frac{1}{q})}\|u\|_{L^p(\Omega)}^{1-d(\frac{1}{p}-\frac{1}{q})}.
\end{equation}
Assertions \eqref{thm2-st2} and \eqref{thm2-st3} now follow from \eqref{thm2-st1},  \eqref{Z1}, \eqref{Z1'}, and H\"older's inequality. 
 
 We finally establish  \eqref{thm2-st4}. Let 
 $$
 \begin{array}{cccc}
{\mathcal T}_\delta : &  [L^2(\Omega)]^2 &  \to &  [L^2(\Omega)]^2 \\[6pt]
 & g & \mapsto & u,  
 \end{array}
 $$
 where $u=(u_1, u_2) \in [H^1(\Omega)]^2$ is the unique solution of \eqref{pro1-sec1'} with $(g_1,g_2) = g$.  We have,  for $q=\frac{p}{p-1}$ and $p>d$, 
$$
\| u \|_{L^q(\Omega)} = \sup_{f \in [L^2(\Omega)]^2; \|f \|_{L^p(\Omega)} \le 1} |\langle u, f \rangle|,
$$
and 
$$
\langle u, f \rangle  =\langle \mathcal{T}_{\lambda}(g), f \rangle =   \langle g, \mathcal{T}_\lambda^* (f) \rangle.
$$
One can check that 
\begin{equation}\label{rem-AD}
\mathcal{T}_{\lambda}^*=\begin{pmatrix}1 & 0\\0 & -1 \end{pmatrix}\mathcal{T}_{\overline{\lambda }}\begin{pmatrix}1 & 0\\0 & -1 \end{pmatrix}. \end{equation}
It follows that 
$$
|\langle u, f \rangle|  \le \| g\|_{L^1(\Omega)} \|\mathcal{T}_{\bar \lambda} (f)\|_{L^\infty(\Omega)} \mathop{\le}^{\eqref{thm2-st3}} C |\lambda|^{-1+\frac{d}{2p}} \| g\|_{L^1(\Omega)} \| f \|_{L^p(\Omega)}.  
$$
Assertion \eqref{thm2-st4} follows. 

It remains to prove the well-posedness of \eqref{pro1-sec1'}. It is clear that the uniqueness of \eqref{pro1-sec1'} follows from \eqref{thm2-st1}. To establish the existence for \eqref{pro1-sec1'}, we use the principle of limiting absorption and the Fredholm theory  as in \cite[the proof of Proposition 4.2]{Cakoni-Ng20} (see also \cite{MinhHung1}). We only consider the case where $\Im(\lambda) < 0$; the other case can be proved similarly. For $\delta >0$,  by the Lax-Milgram theory, there exists a unique solution $ v_\delta = (v_{1, \delta}, v_{2, \delta}) \in [H^1(\Omega)]^2$ of the system 
\begin{equation*}
		\left\{ \begin{array}{cl}
			\dive \big( (1 - i \delta) A_1\nabla v_{1, \delta} \big) - \lambda \Sigma_1 v_{1, \delta}=g_1 ~~&\text{ in}~\Omega,\\[6pt]
			\dive \big( (1+ i \delta) A_2\nabla v_{2, \delta} \big) + \lambda \Sigma_2 v_{2, \delta} = g_2 ~~&\text{ in}~\Omega,\\[6pt]
			v_{2, \delta} - v_{1, \delta}= 0,~~ \big((1 + i \delta) A_2\nabla v_{2, \delta}- (1 - i \delta) A_1\nabla v_{1, \delta} \big) \cdot \nu =0 & \text{ on } \Gamma. 
		\end{array} \right. 
\end{equation*}
Moreover, by \Cref{cor1,cor2},  applied with $G_1 = i \delta A_1 \nabla v_{1, \delta}$ and $G_2 = - i \delta A_2 \nabla v_{2, \delta}$, we have,  for  sufficiently small $\delta$, 
$$
\| \nabla u_\delta \|_{L^2(\Omega)} +  |\lambda|^{1/2} \|  u_\delta \|_{L^2(\Omega)}  \le C |\lambda|^{-1/2} \| g\|_{L^2(\Omega)}. 
$$
By taking $\delta \to 0_+$, one derives the existence of a solution 
$ v = (v_{1}, v_{2}) \in [H^1(\Omega)]^2$ of the system, with $\hSigma_2 = - \Sigma_2$ 
\begin{equation}\label{thm2-sys-hcT}
		\left\{ \begin{array}{cl}
			\dive \big( A_1\nabla v_{1} \big) - \lambda \Sigma_1 v_{1}=g_1 ~~&\text{ in}~\Omega,\\[6pt]
			\dive \big( A_2\nabla v_{2} \big)  -  \lambda \hSigma_2 v_{2} = g_2 ~~&\text{ in}~\Omega,\\[6pt]
			v_{2} - v_{1}= 0,~~ \big(A_2\nabla v_{2} -  A_1\nabla v_{1} \big) \cdot \nu =0 & \text{ on } \Gamma, 
		\end{array} \right. 
\end{equation}
which satisfies 
\begin{equation}\label{thm2-est-hcT}
\| \nabla v \|_{L^2(\Omega)} +  |\lambda|^{1/2} \|  v  \|_{L^2(\Omega)}  \le C |\lambda|^{-1/2} \| g\|_{L^2(\Omega)}. 
\end{equation}
The uniqueness of \eqref{thm2-sys-hcT} is again a consequence of \Cref{cor1} and \Cref{cor2}. 

Define 
\begin{equation*}
\begin{array}{cccc}
\hcT : & [L^2(\Omega)]^2  & \to &  [L^2(\Omega)]^2 \\[6pt]
& g & \mapsto & v, 
\end{array}
\end{equation*}
where $v = (v_1, v_2) \in [H^1(\Omega)]^2$ is the unique solution of \eqref{thm2-sys-hcT}. It follows from \eqref{thm2-est-hcT} that  $\hcT$ is compact. 

It is clear that $u = (u_1, u_2) \in [H^1(\Omega)]^2$ is a solution of \eqref{pro1-sec1'} if and only if 
\begin{equation*}
		\left\{ \begin{array}{cl}
			\dive (A_1\nabla u_1) -\lambda\Sigma_1 u_1=g_1 ~~&\text{ in}~\Omega,\\[6pt]
			\dive (A_2\nabla u_2) - \lambda \hSigma_2 u_2= g_2 + 2 \lambda \Sigma_2 u_2~~&\text{ in}~\Omega,\\[6pt]
			u_1-u_2= 0,~~ (A_1\nabla u_1-A_2\nabla u_2) \cdot \nu =0 & \text{ on } \Gamma. 
		\end{array} \right. 
	\end{equation*}
In other words, 
$$
(u_1, u_2) = \hcT (g_1, g_2) + \hcT (0, 2 \lambda \Sigma_2 u_2). 
$$
Since  this equation has at most one solution and $\hcT$ is compact, this equation has a unique solution by the Fredholm theory. 
The proof is complete. 
\qed
 
\begin{remark} \rm Note from \eqref{rem-AD} that ${\mathcal T}_\lambda$ is not self-adjoint. 
\end{remark}

 \section{Hilbert-Schmidt operators} \label{sect-HS}

We now devote two subsections to the applications of Hilbert-Schmidt operators for the transmission problem. In the first subsection, we recall some basis facts on Hilbert-Schmidt operators and the finite double norms. In the second subsection, we derive their applications for the transmission problem. The main result here is \Cref{pro-trace0}. 

\subsection{Some basic facts on Hilbert-Schmidt operators} 
 In this section, we recall the definition and several properties of Hilbert-Schmidt operators.  We begin with 
 
 \begin{definition} Let $H$ be a separable Hilbert space and let $(\phi_k)_{k=1}^\infty$ be an orthogonal basis. A bounded linear operator $\mathbf{T}: H \to H$ is Hilbert Schmidt  if its finite double norm 
 \begin{equation}\label{X1}
 \vvvert
 \mathbf{T}  \vvvert: =\left(\sum_{k=1}^{\infty} \| \mathbf{T}(\phi_k)\|_{H}^2\right)^{1/2} < + \infty. 
 \end{equation}
The trace  of $\mathbf{T}$ is then defined by 
 \begin{equation}\label{X2}
 \tr (\mathbf{T})=\sum_{k=1}^{\infty}\langle \mathbf{T}(\phi_k),\phi_k\rangle. 
 \end{equation}
 \end{definition}

\begin{remark} \rm The definition of $ \vvvert \bT  \vvvert$ and of  $\tr (\bT)$ do not depend on the choice of $(\phi_k)$, see e.g. \cite[Chapter 12]{Agmon}. 
\end{remark}

One can check,  see \cite[Theorem 12.12]{Agmon},  that if $\bT_1$ and $\bT_2$ are Hilbert Schmidt then $\bT_1 \bT_2$ is also Hilbert Schmidt,  and 
\begin{equation}\label{X3}
\left|\tr (\mathbf{T}_1\mathbf{T}_2)\right|\leq \vvvert \mathbf{T}_1 \vvvert \vvvert \mathbf{T}_2 \vvvert .
\end{equation}

Let $m \in \mN$ and  $\bT : [L^2(\Omega)]^m \to [L^2(\Omega)]^m$ be a Hilbert Schmidt operator. There exists a unique kernel  $\bK \in [L^2(\Omega \times \Omega)]^{m \times m}$, see e.g. \cite[Theorems 12.18 and  12.19]{Agmon},  such that 
\begin{align}\label{Z2}
	(	\mathbf{T}u)(x)=\langle \bK(x,.), u \rangle ~~~\mbox{ for a.e. } \, x\in \Omega, \mbox{ for all } u \in [L^2(\Omega)]^m. 
	\end{align}
Moreover, 
\begin{equation}\label{norm-K-T}
\vvvert\mathbf{T} \vvvert^2=\mathop{\iint}_{\Omega\times\Omega}|\bK(x,y)|^2 \,  dx \,  dy. 
\end{equation}
Note that \cite[Theorems 2.18 and  12.19]{Agmon} state for  $m=1$,  nevertheless, the same arguments hold for $m \in \mN$. 

\medskip 
 We have 
\begin{lemma}\label{lem-HS1} Let $d \ge 2$, $m \in \mN$,  and  $\bT : [L^2(\Omega)]^m \to [L^2(\Omega)]^m$ be  such that $\mathbf{T}(\phi)\in \mathbf{C}(\bar \Omega)$ for $\varphi  \in [L^2(\Omega)]^m$,  and 
	\begin{equation}\label{lem-HS1-st0}
	\|\mathbf{T}(\phi) \|_{L^\infty (\Omega)}\leq M \|\phi \|_{L^2 (\Omega)}, 
\end{equation} 
for some $M \ge 0$. 
Then  $\mathbf{T}$ is a Hilbert-Schmidt operator,  
	\begin{align}\label{lem-HS1-st1}
\vvvert \mathbf{T} \vvvert \leq |\Omega|^{1/2}M,
	\end{align}
	and  the kernel $\bK$ of $\bT$ satisfies
	\begin{equation}
	\label{lem-HS1-st1'}
	\sup_{x\in \Omega}\left(\int_{\Omega}|\bK(x,y)|^2  dy\right)^{1/2}\leq  |\Omega|^{1/2}M.
	\end{equation}
Assume in addition that 
		\begin{equation}\label{lem-HS1-st2}
	\| \mathbf{T}(\phi) \|_{L^\infty (\Omega)}\leq \widetilde{M}||\phi||_{L^1 (\Omega)} \mbox{ for } \phi\in [L^2(\Omega)]^m, 
	\end{equation}
	for some $\widetilde{M} \ge 0$, 
then the kernel $\bK$ of $\bT$ satisfies 
	\begin{align}\label{lem-HS1-st3}
	|\bK(x,y)|\leq \widetilde{M} ~~\forall x, y\in \Omega.
	\end{align}
\end{lemma}

\begin{proof} The proof is quite standard as in \cite{Agmon}. 
We present the details of this proof for the convenience of the reader. Let $(\phi_k)_{k=1}^\infty$ be an orthonormal basis of $[L^2(\Omega)]^m$ and set $\varphi_j=\mathbf{T}(\phi_j)$. Let $a_1,...,a_N \in \mC$ be arbitrary. By \eqref{lem-HS1-st0}, we have 
$$
\Big| \sum_{j=1}^{N}a_j\varphi_j(x) \Big| \leq M \| \sum_{j=1}^{N}a_j\phi_j \|_{L^2 (\Omega)}=M\left(\sum_{j=1}^{N}|a_j|^2\right)^{1/2}~~\forall x\in \Omega.$$
Choosing $a_j=\overline{\varphi_j(x)}$ yields 
$$
\sum_{j=1}^{N}| \varphi_j(x)|^2 \leq  M^2~~\forall x\in \Omega.
$$
Integrating over $\Omega$, we obtain $$
\sum_{j=1}^{N} \| \varphi^j \|^2_{L^2(\Omega)} \leq |\Omega| M^2, 
$$
which implies \eqref{lem-HS1-st1}.  

Assertion \eqref{lem-HS1-st1'} follows from \eqref{lem-HS1-st0} by  \eqref{Z2}.

It is clear that \eqref{lem-HS1-st3} is a consequence of \eqref{lem-HS1-st2} by the definition of the kernel.
\end{proof}

We next recall a  basic, useful property of a Hilbert-Schmidt operator, see e.g.,  \cite[Theorem 12.21]{Agmon} \footnote{Note that \cite[Theorems 2.21]{Agmon} states for  $m=1$,  nevertheless, the same arguments hold for $m \in \mN$. 
}: 

\begin{lemma} \label{lem-T1T2} Let $m \in \mN$ and let 	$\mathbf{T}_1,\mathbf{T}_2$ be two Hilbert-Schmidt operators in $[L^2(\Omega)]^m$ with the corresponding kernels $\bK_1$ and $\bK_2$. Then $\bT:  = \mathbf{T}_1\mathbf{T}_2$ is a Hilbert-Schmidt operator with the kernel $\mathbf{K}$ given by 
\begin{equation}\label{lem-T1T2-1}
\bK(x, y) = \int_{\Omega} \bK_1 (x, z)  \bK_2 (z, y) \, dz. 
\end{equation}
Moreover, 
\begin{equation}\label{traceT^2}
	\tr (\mathbf{T}_1\mathbf{T}_2)=\int_{\Omega} \tr ~\mathbf{K}(x,x)dx.
	\end{equation}
\end{lemma}

\begin{remark} \rm Using \eqref{lem-T1T2-1}, one can check that 
$$
\int_{\Omega} |\bK(x,x)| \, dx \le \int_{\Omega } |\bK_1(x, z)| |\bK_2(z, x)| \, dz \,dx \le  \| \bK_1 \|_{L^2(\Omega \times \Omega)} \| \bK_2 \|_{L^2(\Omega \times \Omega)}. 
$$
Hence $\bK(x, x) \in [L^1(\Omega)]^{m \times m}$. 
\end{remark}

\subsection{Applications of the theory of Hilbert-Schmidt operators}

In this section, we apply the theory of Hilbert-Schmidt operators 
to the operator $T_\lambda$ mentioned in the introduction. The main ingredient of the analysis is \Cref{thm2}. We begin with

\begin{definition} \label{De1} Let $\eps_0>0$ and  $\Lambda \ge 1$.  Assume the assumptions of \Cref{thm2} hold. Let $\Lambda_0$ and $C$ be the constants in  \Cref{thm2}. For $\lambda \in \cL(\theta, \Lambda_0)$ with $\inf_{n \in \mZ} |\theta -  n \pi| > \eps_0$,  define 
$$
\begin{array}{ccccc}
T_\lambda:  & [L^2(\Omega)]^2 & \to &  [L^2(\Omega)]^2 \\[6pt]
& f &  \mapsto & u, 
\end{array}
$$
where $ u = (u_1, u_2) \in [H^1(\Omega)]^2$ is the unique solution of, with $(f_1, f_2) = f$,  
	\begin{equation}\label{pro1a*} 
	\left\{ \begin{array}{cl}
	\dive(A_1 \nabla u_1) - \lambda\Sigma_1 u_1= \Sigma_1 f_1 &\text{ in}~\Omega, \\[6pt]
	\dive(A_2 \nabla u_2) - \lambda\Sigma_2 u_2= \Sigma_2 f_2 &\text{ in}~\Omega, \\[6pt]
	u_1 =u_2, \quad  A_1 \nabla u_1\cdot \nu = A_2 \nabla u_2\cdot \nu & \text{ on }\Gamma.
	\end{array}
	\right.
	\end{equation}
\end{definition}

From now on, we fix the constant $\Lambda_0$ as required in \Cref{De1} for a given $\eps_0$ and set 
\begin{equation}\label{def-lambda_0}
\lambda_0 = \Lambda_0 e^{i \pi/2}. 
\end{equation}

\begin{remark} \rm  Let $T^\star_\lambda$ be  the adjoint operator of $T_\lambda$ , i.e.,   $
\langle T_{\lambda}(f), g \rangle =   \langle f, T_\lambda^* (g) \rangle
	$  for any $f, \, g \in [L^2(\Omega)]^2 $. Integrating by parts, one has
	\begin{align*}
	&\left(\int_\Omega 	\dive(A_1 \nabla u_1) v_1- 	\dive(A_2 \nabla u_2) v_2\right) -	\left(\int_\Omega 	u_1\dive(A_1 \nabla v_1) - u_2	\dive(A_2 \nabla v_2)\right) \\&=\left(\int_{\Gamma} v_1.	A_1 \nabla u_1\cdot\nu- 	v_2. A_2 \nabla u_2\cdot\nu\right)-\left(\int_{\Gamma} u_1	A_1 \nabla v_1\cdot\nu - 	u_2A_2 \nabla v_2\cdot\nu\right).
	\end{align*}  
	This implies 
	\begin{equation}\label{adjopT}
		T_{\lambda}^*=\begin{pmatrix}\Sigma_1 & 0\\0 & -\Sigma_2 \end{pmatrix}T_{\overline{\lambda }}\begin{pmatrix}1/\Sigma_1 & 0\\0 & -1/\Sigma_2 \end{pmatrix}. 
	\end{equation} 
Thus $T_\lambda$ is not self-adjoint. 
\end{remark}
For the operator $T_\lambda$ defined above, the following estimates hold:  

\begin{proposition}\label{pro-HS}   We have
	\begin{equation}\label{Z3}
	\|T_{\lambda} \|_{L^p \to L^\infty}\leq C  |\lambda|^{-1+\frac{d}{2p}} ~~\text{if}~~p>d,
	\end{equation}
		\begin{equation}\label{Z4}
	\|T_{\lambda} \|_{L^p \to L^{q}}\leq C   |\lambda|^{-1+\frac{d}{2}\left(\frac{1}{p}-\frac{1}{q}\right)}~~\text{if}~~1<p < d, ~~p\leq q< \frac{dp}{d-p},
	\end{equation}
	\begin{equation}\label{Z5}
	\|T_{\lambda} \|_{L^1 \to L^q} \le C  |\lambda|^{-1+\frac{d}{2}-\frac{d}{2q}}~~\text{if}~~1<q<\frac{d}{d-1}.
	\end{equation}
Assume that  $\lambda_{1},...,\lambda_{k+1}$ satisfy the assumption of \Cref{thm2} with  $k= k_d = \left[\frac{d}{2}\right]$,  and $|\lambda_1|\sim |\lambda_2|\sim...\sim |\lambda_{k+1}|=t$.  Then  operator $\prod_{j=1}^{k+1}T_{\lambda_j}=T_{\lambda_{k+1}}\circ T_{\lambda_{k}}\circ...\circ T_{\lambda_{1}}$ is Hilbert-Schmidt, and 
		\begin{equation}\label{Z6}
		\Big\vvvert \prod_{j=1}^{k+1}T_{\lambda_j} \Big\vvvert \leq C t^{\frac{d}{4}-1-k}.
		\end{equation}

\end{proposition}

\begin{proof} Clearly, \eqref{Z3} - \eqref{Z5} follow from \eqref{thm2-st2} - \eqref{thm2-st4}. Fix  $p_1=2<p_2<...<p_{k}<p_{k+1}< + \infty$  with  $p_{k+1}>d$ and $p_{j} <  \frac{dp_{j-1}}{d-p_{j-1}}$. By \eqref{Z3} and \eqref{Z4}, we obtain 
	\begin{multline*}
	\Big\|\prod_{j=1}^{k+1}T_{\lambda_j} \Big\|_{L^2\to L^\infty}\leq 	\|T_{\lambda_1}\|_{L^{p_1}\to L^{p_2}} ....	\|T_{\lambda_k}\|_{L^{p_k}\to L^{p_{k+1}}}\|T_{\lambda_{k+1}}\|_{L^{p_{k+1}}\to L^{\infty}} \\[6pt]
	\leq C  \left( \prod_{j=1}^{k} 
|\lambda|^{-1 + \frac{d}{2} \Big(\frac{1}{p_j} - \frac{1}{p_{j+1}} \Big)} \right) |\lambda|^{-1 + \frac{d}{2 p_{k+1}}} = C |\lambda|^{\frac{d}{4}  - (k+1)}.
	\end{multline*}
	The conclusions now follow from \Cref{lem-HS1}. 
\end{proof}

The following is the main result of this section and plays a crucial role in our analysis. 

\begin{proposition}\label{pro-trace0} Let $k= k_d = \left[\frac{d}{2}\right]$ and denote
$$
\theta_j = \Big( \frac{1}{4}  + 2 (j-1) \Big) \frac{\pi}{k+1}  \mbox{ and }~~ \theta_{k+1 + j} = \Big( \frac{5}{4}  + 2 (j-1) \Big) \frac{\pi}{k+1}\quad \mbox{ for } 1 \le j \le k+1.
$$
Let  $t > 10 \Lambda_0$  and set $\mu_j = \lambda_0 + t z_j$ with $z_j  = e^{i \theta_j}$  for $j=1,\dots ,2(k+1)$. We have 
\begin{align}\label{equ-trace1-sec1} 
	\tr ~(T_{\mu_{2(k+1)}}\circ T_{\mu_{2k+1}}\circ...\circ T_{\mu_{1}} )=\sum_{j}\frac{1}{\tlambda_j^{2(k+1)} -it^{2(k+1)}},
\end{align}
	where each  characteristic value $\tlambda_j$ of $T_{\lambda_0}$ is repeated  a number of times equal to its multiplicity. 
\end{proposition}

\begin{remark} \rm In \Cref{pro-trace0}, $\Lambda_0$ is chosen large and corresponds  with $\eps_0 = \frac{\pi}{8(k+1)}$. 

\end{remark}

\begin{proof} It is clear that  $z_{1}, \dots , z_{k+1}$ are the solutions of $z^{k+1}-e^{i\frac{\pi}{4}}=0$ in $\mathbb{C}$ and $z_{k+2}, \dots , z_{2(k+1)}$ are the solutions of $z^{k+1}-e^{i\frac{5\pi}{4}}=0$ in $\mathbb{C}$. One then has, for $z\in \mathbb{C}$,  
\begin{equation}\label{prepare1}
\prod_{j=1}^{k+1}(z-z_{j})   =  z^{k+1}-e^{i\frac{\pi}{4}}, \quad
\prod_{j=1}^{k+1}(1-z_{j}z)  =1-e^{i\frac{\pi}{4}}z^{k+1}
\end{equation}
and 
\begin{equation}\label{prepare1-1}
\prod_{j=k+2}^{2(k+1)}(z-z_{j})   =  z^{k+1}-e^{i\frac{5\pi}{4}}, \quad
\prod_{j=k+2}^{2(k+1)}(1-z_{j}z)  =1-e^{i\frac{5\pi}{4}}z^{k+1}. 
\end{equation}

Note that, if $T_\lambda$ and $T_{\lambda+s}$ exist, and $T_\lambda$ is compact, then $s$ is not a characteristic value of $T_\lambda$, and 
\begin{equation}\label{prepare2}
T_{\lambda+s} = T_{\lambda }(I-sT_{\lambda} )^{-1} = (I-sT_{\lambda} )^{-1} T_{\lambda }.  
\end{equation}
Indeed, if $T_\lambda$ and $T_{\lambda+s}$ exist, one can check that $I - s T_\lambda$ is injective, and therefore subjective since $T_\lambda$ is compact. 
One can then show that \eqref{prepare2} holds.

As a consequence of \eqref{prepare2}, $T_{\lambda+s}$ is the modified operator of $T_\lambda$ with respect to $s$. 

Set  
$$
	\mathbf{T}=T_{\mu_{2(k+1)}}\circ...\circ T_{\mu_{k+2}}.
$$ 
It follows from \Cref{pro-HS} that  $\mathbf{T}$ is Hilbert-Schmidt,  and 
$$
	\vvvert \mathbf{T}\vvvert  \leq C t^{\frac{d}{4}-1-k}, \quad \mbox{ and } \quad \|\mathbf{T}\|_{L^2\to L^2}\leq C t^{-k-1}.
$$
Let $s_1, \, s_2, \dots$ be the characteristic values  of $\mathbf{T}$ repeated  a number of times equal to their multiplicities.   Thanks to \cite[Theorem 12.17]{Agmon}, one has, for a non-characteristic value $\lambda$ of $\bT$,  
	 \begin{align}\label{eq-A}
	 \tr ~(\mathbf{T}\circ (\mathbf{T})_{\lambda}  )=\sum_{j}\frac{1}{s_j(s_j-\lambda)} + c_t,
	 \end{align} 
where $(\bT)_{\lambda}$ is the modified operator associated with $\bT$ and  $\lambda$, i.e., $(\mathbf{T})_{\lambda}: =\mathbf{T}(I-\lambda\mathbf{T})^{-1}$, for some $c_t \in \mC$. 
	 
By applying \eqref{eq-A} with $\lambda = 2e^{\frac{i\pi}{4}}t^{k+1}$, it suffices to establish 
	\begin{equation}\label{Z10}
	(\mathbf{T})_{\lambda} =T_{\mu_{k+1}}\circ...\circ T_{\mu_1} \mbox{ for } \lambda = 2e^{\frac{i\pi}{4}}t^{k+1}, 
	\end{equation}
	\begin{equation}\label{Z9}
	s_j=\tlambda_\ell^{k+1}-e^{i\frac{5\pi}{4}}t^{k+1},
	\end{equation}
	for some $\ell$, and 
	\begin{equation}\label{Z9-10}
	\mbox{ the multiplicity of $s_j$ is equal to the sum of  the multiplicity of $\tlambda_\ell$ such that \eqref{Z9} holds}, 
	\end{equation}
	and 
	\begin{equation}\label{AZ}
	c_t = 0. 
	\end{equation}

This will be done in the next three steps. 
	
\medskip 	
\noindent \underline{Step 1: }  Proof of  \eqref{Z10}. Since $\mu_j -\lambda_0=z_j t$, it follows from  the second identity in \eqref{prepare1} that  
\begin{equation}\label{prepare3}
	\prod_{l=1}^{k+1}\big(1-(\mu_l-\lambda_0)z \big)=1-e^{i\frac{\pi}{4}}t^{k+1}z^{k+1}. 
\end{equation}
One has
\begin{multline}\label{TTT}
	T_{\mu_{k+1}}\circ...\circ T_{\mu_1} \mathop{=}^{\eqref{prepare2}} T_{\lambda_0}\left(I-(\mu_{k+1}-\lambda_0)T_{\lambda_0}\right)^{-1}\circ...\circ T_{\lambda_0}\left(I-(\mu_{1}-\lambda_0)T_{\lambda_0}\right)^{-1} \\[6pt]
	\mathop{=}^{\eqref{prepare2}} 
	T_{\lambda_0}^{k+1} \prod_{l=1}^{k+1}(I-(\mu_l-\lambda_0) T_{\lambda_0} )^{-1}
	\mathop{=}^{\eqref{prepare3}} T_{\lambda_0}^{k+1}\left(I-e^{i\frac{\pi}{4}}t^{k+1}T_{\lambda_0}^{k+1}\right)^{-1}.
\end{multline}
In other words, we have 
\begin{equation}\label{pro-trace-p1}
	T_{\mu_{k+1}}\circ...\circ T_{\mu_1}=\left(T_{\lambda_0}^{k+1}\right)_{e^{i\frac{\pi}{4}}t^{k+1}}.
\end{equation}	

Similarly, we obtain  
\begin{equation}\label{pro-trace-p2}
	\bT = T_{\mu_{2(k+1)}}\circ...\circ T_{\mu_{k+2}} =\left(T_{\lambda_0}^{k+1}\right)_{e^{i\frac{5\pi}{4}}t^{k+1}}.
\end{equation}	
Using the property 
$$
\Big(\left(T_{\lambda_0}^{k+1}\right)_{\gamma_1} \Big)_{\gamma_2} = \left(T_{\lambda_0}^{k+1}\right)_{\gamma_1 + \gamma_2},
$$
for $\gamma_1$ and $\gamma_1 + \gamma_2$  non-characteristic values of $T_{\lambda_0}^{k+1}$, we derive from \eqref{pro-trace-p1} and \eqref{pro-trace-p2} that 
$$
	(\mathbf{T})_\lambda=\left(T_{\lambda_0}^{k+1}\right)_{e^{i\frac{5\pi}{4}}t^{k+1}+\lambda}=\left(T_{\lambda_0}^{k+1}\right)_{e^{i\frac{\pi}{4}}t^{k+1}} = T_{\mu_{k+1}}\circ...\circ T_{\mu_1}, 
$$
and  \eqref{Z10} follows. 

	\medskip
		\noindent \underline{Step 2: }  Proof of \eqref{Z9} and \eqref{Z9-10}. Since $\bT = (T_{\lambda_0}^{k+1})_{e^{i 5 \pi/4} t^{k+1} }$, it follows,  see e.g. \cite[Theorem 12.4]{Agmon}, that $s_j^{-1}$ is an eigenvalue of $\bT$ that  is  not equal to $  - e^{-i 5 \pi/4} t^{-(k+1)} $ if and only if $\frac{s_j^{-1}}{1 + s_j^{-1}  e^{i 5 \pi/4} t^{k+1}} = \frac{1}{s_j +  e^{i 5 \pi/4} t^{k+1}}$ is an eigenvalue of $T_{\lambda_0}^{k+1}$ (or equivalently $s_j + e^{i 5 \pi/4} t^{k+1} $ is a characteristic value of $T_{\lambda_0}^{k+1}$),  and they have the same multiplicity. One can check that  $- e^{-i 5 \pi/4} t^{-(k+1)} $ is not an eigenvalue of $\bT$.  Assertions \eqref{Z9} and \eqref{Z9-10} follow. 

\medskip
		\noindent \underline{Step 3: }  Proof of \eqref{AZ}. For $z \in \cL (\theta, 1)$ with $\inf_{n \in \mZ}|\theta - n \pi| > \eps_0$ and $|z|$ large enough, let $\tau_1, \cdots, \tau_{k+1}$ be the $k+1$ distinct roots in $\mC$ of the equation $x^{k+1} = z$. Set 
$$
\eta_l = \lambda_0 + \tau_l \mbox{ for } 1 \le l \le k+1.  
$$
As in the proof of \eqref{TTT}, one has 
$$
T_{\eta_{k+1}} \circ \dots \circ T_{\eta_1} = T_{\lambda_0}^{k+1} \Big(I - z T_{\lambda_0}^{k+1} \Big)^{-1}. 
$$
It follows that 
$$
T_{\eta_{k+1}} \circ \dots \circ T_{\eta_1} = \Big(T_{\lambda_0}^{k+1}\Big)_{z}. 
$$

Consider $\lambda$ defined by $e^{i\frac{5\pi}{4}}t^{k+1}+\lambda = z$.  We have, for  large $|z|$,  
\begin{equation}\label{lem-trace-p1}
	| \tr ~(\mathbf{T}\circ (\mathbf{T})_{\lambda}  ) | \mathop{\leq}^{\eqref{X3}} \vvvert {\bf T}  \vvvert  \vvvert {\bf T}_{\lambda} \vvvert \mathop{\leq}^{\eqref{Z6}} C_t | z|^{\frac{d}{4}-1 - k}\to 0~~\text{as}~~ |z| \to + \infty,
\end{equation}
	and 
	\begin{equation}\label{lem-trace-p2}
	\left|\sum_{j}\frac{1}{s_j(s_j- \lambda)} \right|\leq \left(\sum_{j} |s_j|^{-2} \right)^{1/2} \left(\sum_{j}|s_j - \lambda |^{-2} \right)^{1/2}. 
	\end{equation}
Applying  \cite[Theorem 12.14]{Agmon}, we have 
\begin{equation}\label{lem-trace-p3-1}
\sum_{j} |s_j|^{-2} \le \vvvert {\bf T} \vvvert \le C_t, 
\end{equation}	
and applying  \cite[Theorems 12.4 and 12.14]{Agmon}, we obtain 
\begin{equation}\label{lem-trace-p3-2}
\sum_{j}|s_j - \lambda |^{-2}  \le \vvvert {\bf T}_{\lambda} \vvvert \mathop{\le}^{\eqref{Z6}}
C_t | z|^{\frac{d}{4}-1 - k}\to 0~~\text{as}~~ |z| \to + \infty. 
\end{equation}	
We derive from \eqref{lem-trace-p2}, \eqref{lem-trace-p3-1}, and  \eqref{lem-trace-p3-2} that 
\begin{equation}\label{lem-trace-p4}
\sum_{j}\frac{1}{s_j(s_j-\lambda)}  \to 0 \mbox{ as } |z| \to + \infty. 
\end{equation}
Combining \eqref{lem-trace-p1} and \eqref{lem-trace-p4} yields 
$c_t = 0$.	

\medskip
The proof is complete. 
\end{proof}

\begin{remark} \label{rem-multiplicity} \rm Let $\lambda_j$ be an eigenvalue of the transmission problem. 
Then $\lambda_j - \lambda$ and $\lambda_j - \hat  \lambda$ are the characteristic values of $T_\lambda$ and $T_{\hat \lambda}$ respectively, provided that $T_\lambda$ and $T_{\hat \lambda}$ exist. 
Using \eqref{prepare2} and applying \cite[Theorem 12.4]{Agmon}, one can show that the multiplicity of $\lambda_j - \lambda$  and the multiplicity of  $\lambda_j - \hat \lambda$ are the same. 
\end{remark}

\section{The Weyl law for eigenvalues of the transmission problem - Proof of \Cref{thm1}}\label{sect-WL}

\subsection{Approximation of the trace of the kernel and their applications}  For  $\lambda \in \cL(\theta, 1) $ with $\theta \neq n \pi$ for all $n \in \mZ$,  and $x_0 \in \Omega$, set
\begin{equation}
\begin{array}{ccc}
 S_{j, \lambda, x_0}: L^2(\mR^d)  & \to &  L^2(\mR^d) \\[6pt]
 f_j & \mapsto & v_j, 
\end{array}
\end{equation}
where $v_j \in H^1(\mR^d)$ is the unique solution of 
$$
\dive(A_j(x_0) \nabla v_j) -\lambda\Sigma_j(x_0) v_j=  \Sigma_j(x_0) f_j  \quad \mbox{ in } \mR^d.
$$
We also define 
\begin{equation}
\begin{array}{ccc}
 S_{\lambda, x_0}: [L^2(\mR^d)]^2   & \to &  [L^2(\mR^d)]^2 \\[6pt]
 (f_1,f_2) & \mapsto & (S_{1, \lambda, x_0} f_1, S_{2, \lambda, x_0} f_2).  
\end{array}
\end{equation}
One then has 
 \begin{equation*}
  S_{j,\lambda,x_0}f (x)=\int_{\mR^d}F_{j,\lambda}(x_0,x-y)f_j (y)dy,
 \end{equation*}
where
$$
F_{j,\lambda}(x_0,z)=-\frac{1}{(2\pi)^d}\int_{\mathbb{R}^d}\frac{e^{iz\xi}}{\Sigma_j(x_0)^{-1} \langle  A_j(x_0) \xi, \xi \rangle+\lambda } d\xi~~\text{for}~~z \in \mR^d.$$
\newline

By \Cref{lem1},  we get, for $1 < p < + \infty$, 
\begin{equation}\label{Z8}
\| \nabla^2 S_{j,\lambda,x_0} f_j \|_{L^p(\mathbb{R}^d)} + |\lambda|^{1/2} \| \nabla S_{j,\lambda,x_0} f_j \|_{L^p(\mathbb{R}^d)} +|\lambda| \|S_{j,\lambda,x_0} f_j \|_{L^{p}(\mathbb{R}^d)}\leq C  \| f_j \|_{L^{p}(\mathbb{R}^d)}. 
\end{equation} 
As in the proof of \Cref{thm2}, we obtain from the interpolation inequalities \eqref{Z1} and \eqref{Z1'} that 
	\begin{equation}\label{trace-S1}
\|S_{\lambda,x_0}  \|_{L^p \to L^\infty}\leq C  |\lambda|^{-1+\frac{d}{2p}} ~~\text{if}~~p>d,
\end{equation}
\begin{equation}\label{trace-S2}
\|S_{\lambda,x_0} \|_{L^p \to L^{q}}\leq C   |\lambda|^{-1+\frac{d}{2}\left(\frac{1}{p}-\frac{1}{q}\right)}~~\text{if}~~1<p < d, ~~p\leq q< \frac{dp}{d-p},
\end{equation}
\begin{equation}\label{trace-S3}
\|S_{\lambda,x_0} \|_{L^1 \to L^q} \le C  |\lambda|^{-1+\frac{d}{2}-\frac{d}{2q}}~~\text{if}~~1<q<\frac{d}{d-1}.
\end{equation}

Let  $t > 10 \Lambda_0$ and  let $\mu_{1}, \dots , \mu_{2(k+1)}$ be defined in \Cref{pro-trace0}. Set,  for $z \in \mR^d$, 
\begin{align}\label{Z7-1}
\cF_{j,t}(x_0,z)= \frac{1}{(2\pi)^d}\int_{\mathbb{R}^d}\frac{e^{iz\xi}}{ \prod_{l=1}^{2 (k+1)}  \Big(\Sigma_j(x_0)^{-1}  \langle  A_j(x_0) \xi, \xi \rangle+ \mu_l \Big)} \, d\xi,  
\end{align} 
and define
$$
\cS_{j, t, x_0}   = \prod_{l=1}^{2(k+1)}S_{j,\mu_l ,x_0}. 
$$
Then 
\begin{equation*}
 \cS_{j, t, x_0} f_j (x)= \int_{\mR^d}\cF_{j,t}(x_0,x-y) f_j (y)dy.
 \end{equation*}
Since, by the definition of $\mu_l$ and  $z_l$, 
\begin{multline*}
\prod_{l=1}^{2 (k+1)}  \Big( \Sigma_j(x_0)^{-1} \langle  A_j(x_0) \xi, \xi \rangle+ \mu_l \Big) = \prod_{l=1}^{2 (k+1)}  \Big(\Sigma_j(x_0)^{-1}  \langle  A_j(x_0) \xi, \xi \rangle + \lambda_0 + t z_l \Big) \\[6pt]
= \big(\Sigma_j(x_0)^{-1}  \langle  A_j(x_0) \xi, \xi \rangle + \lambda_0 \big)^{2 (k+1)}  - i t^{2 (k+1)}, 
\end{multline*}
it follows from \eqref{Z7-1} that 
\begin{align}\label{Z7-11}
\cF_{j,t}(x_0,z)= \frac{1}{(2\pi)^d}\int_{\mathbb{R}^d}\frac{e^{iz\xi}}{ \big(\Sigma_j^{-1}(x_0) \langle  A_j(x_0) \xi, \xi \rangle + \lambda_0)^{2(k+1)}-i t^{2(k+1)}  } \, d\xi. 
\end{align} 

As a consequence of \eqref{Z7-11}, we obtain, by a change of variables,  
\begin{multline}
\cF_{j,t}(x_0,0)= \frac{1}{(2\pi)^d}\int_{\mathbb{R}^d}\frac{1}{ \big(\Sigma_j^{-1}(x_0) \langle  A_j(x_0) \xi, \xi \rangle + \lambda_0 \big)^{2(k+1)}-i t^{2(k+1)}  } \, d \xi  \\[6pt] = 
\frac{t^{\frac{d}{2} - 2(k+1)}}{(2\pi)^d}\int_{\mathbb{R}^d}\frac{1}{ \big(\Sigma_j^{-1}(x_0) \langle  A_j(x_0) \xi, \xi \rangle + t^{-1}\lambda_0 \big)^{2(k+1)} - i }   \, d \xi. 
\end{multline} 
This implies, by the dominated convergence theorem, 
\begin{align}\label{Z7}
\cF_{j,t}(x_0,0)= 
\frac{t^{\frac{d}{2} - 2(k+1)}}{(2\pi)^d}\int_{\mathbb{R}^d}\frac{1}{ \big(\Sigma_j^{-1}(x_0) \langle  A_j(x_0) \xi, \xi \rangle \big)^{2(k+1)}-i    } \, d \xi + \mathcal{O}(t^{\frac{d}{2} - 2 (k+1)-1}). 
\end{align} 

We next introduce $\cS_{t,x_0}: [L^2(\mR^d)]^2 \to [L^2(\mR^d)]^2$ by 
$$
\cS_{t,x_0} = S_{\mu_{2(k+1)} , x_0}\circ \cdots \circ S_{\mu_{1},x_0}, 
\quad \mbox{ where } \quad S_{\mu_l , x_0} = \begin{pmatrix} S_{1, \mu_l ,x_0} & 0\\0 & S_{2, \mu_l ,x_0} 
\end{pmatrix}. 
$$
Set 
$$
\cF _{t}(x_0, \cdot )=\begin{pmatrix}\cF_{1,t}(x_0, \cdot )& 0\\0 & \cF_{2,t}(x_0, \cdot ) \end{pmatrix}.
$$
We then have 
\begin{equation*}
 \cS_{t, x_0} f (x)= \int_{\mR^d}\cF_{t}(x_0,x-y) f (y)dy. 
 \end{equation*}

Let $\cK_{t}$ denote the kernel corresponding to $\prod_{l=1}^{2(k+1)}T_{\mu_l}=T_{\mu_{2(k+1)}}\circ \cdots \circ T_{\mu_{1}}$. Here is the main result of this section.

\begin{proposition} \label{pro-trace}  We have 
	\begin{equation}\label{trace-1-sec1}
\int _{\Omega}\tr ~\cK_{t}(x,x)dx= \mathbf{\hat c} t^{\frac{d}{2}-2(k+1)} +o(t^{\frac{d}{2}-2(k+1)})~~\text{as}~~t\to\infty,
	\end{equation}
	where 
\begin{equation}\label{def-hc}
	\mathbf{\hat c}=  \frac{1}{(2\pi)^d} \sum_{j=1}^2  \int_\Omega \int_{\mathbb{R}^d}\frac{1}{ \big(\Sigma_j^{-1}(x) \langle  A_j(x) \xi, \xi \rangle)^{2(k+1)}-i  } \, d\xi \, dx.  
\end{equation}
\end{proposition}

\begin{proof}  
We claim that 
\begin{equation}\label{trace-1-sec1-00}
\int _{\Omega}\tr ~\cK_{t}(x,x)dx=   \int _{\Omega}\tr ~\cF_{t}(x,0)dx +o(t^{\frac{d}{2}-2(k+1)})~~\text{as}~~t\to\infty. 
	\end{equation}
The conclusion then follows from \eqref{Z7}. 
	
The main point of the  proof  is to establish  \eqref{trace-1-sec1-00}.  Let $\varphi$ be a function in $C^\infty (\mathbb{R}^d)$ such that $0\leq \varphi\leq 1$, $\varphi=1$ in $B_{\frac{1}{2}}$ and $\supp \varphi\subset B_{\frac{1}{2}+\frac{1}{100d}}$. Let $\delta_0>0$ and  $x_0\in \Omega$ be such that $\dist(x_0,\partial\Omega)>\delta_0$.
For $\delta\in (0, 10^{-2}\delta_0)$, set $\varphi_{\delta}(x)=\varphi(\delta^{-1}(x-x_0))$ and 
$$
\Phi(\delta,x_0)=\sup_{B_{10\delta}(x_0)}\sum_{j=1}^2 \left(|A_j(x)-A_j(x_0)|+|\Sigma_j(x)-\Sigma_j(x_0)|\right).
$$
The essential ingredient of  the analysis is the following estimate, for $t > \delta^{-4}$: 
\begin{equation}\label{X00}
\| \varphi_{2\delta}\left(T_{\mu_{2(k+1)}}\circ...\circ T_{\mu_{1}}-  \cS_{t, x_0} \right)\varphi_{\delta}
\|_{L^{1}\to L^\infty} \le C_{\delta_0} \left(\Phi(\delta,x_0) +\delta^{-1} t^{-1/2}\right) t^{\frac{d}{2}-2(k+1)}. 
\end{equation}

We first assume \eqref{X00} and continue the proof. We have 
$$
\Big(\varphi_{2\delta}\left(T_{\mu_{2(k+1)}}\circ...\circ T_{\mu_{1}}- \cS_{t, x_0}\right)\varphi_{\delta} \Big) (f) (x) = \varphi_{2\delta}(x) \int_\Omega \Big( \cK_{t}(x,y) - \cF_{t} (x_0, x-y) \Big) \varphi_{\delta}(y)  f(y) dy. 
$$
It follows from  \eqref{X00} that, for $x, y \in \Omega$ and  for $t > \delta^{-4}$, 
\begin{align*}
\Big| \varphi_{2\delta}(x) \varphi_{\delta}(y) \big(  \cK_{t}(x,y)- \cF_{t}(x_0, x-y)\big)\Big|~\leq C_{\delta_0}\left(\Phi(\delta,x_0) +\delta^{-1}t^{-1/2}\right)t^{\frac{d}{2}-2(k+1)}. 
\end{align*}
This implies that, for $t > \delta^{-4}$,  
\begin{equation}\label{trace-sect1-p1}
\left| \tr ~\cK_{t}(x_0,x_0)- \tr~\cF_{t}(x_0,0)\right|\leq C_{\delta_0}\left(\Phi(\delta,x_0) +\delta^{-1}t^{-1/2}\right)t^{\frac{d}{2}-2(k+1)}. 
\end{equation}
Here we used the fact that $\varphi_{2 \delta}(x_0) = \varphi_\delta(x_0) = 1$. Using  \Cref{pro-HS} and \eqref{lem-HS1-st2}-\eqref{lem-HS1-st3}, we have 
\begin{equation}\label{trace-sect1-p2}
|\cK_{t}(x, x)| \le C t^{\frac{d}{2} -2(k+1)} \mbox{ for } x\in \Omega.
\end{equation}
By \eqref{Z7}, we obtain 
\begin{equation}\label{trace-sect1-p3}
|\cF_{t}(x, 0)| \le C t^{\frac{d}{2} -2(k+1)} \mbox{ for } x \in \Omega. 
\end{equation}
Assertion~\eqref{trace-1-sec1-00} now follows from  \eqref{trace-sect1-p1}, \eqref{trace-sect1-p2}, and \eqref{trace-sect1-p3} by noting that $\sup_{x \in \Omega}\Phi(\delta, x) \to 0$ as $\delta \to 0$.

\medskip It remains to prove \eqref{X00}.   We have 
\begin{multline}\label{trace-sect1-a1}
\varphi_{2\delta}\left(T_{\mu_{2(k+1)}}\circ \dots \circ T_{\mu_{1}}-  \cS_{t, x_0} \right)\varphi_{\delta} \\[6pt]
= \sum_{l=1}^{2(k+1)} \varphi_{2\delta}\left(T_{\mu_{2(k+1)}}\circ \dots T_{\mu_{l+1}} \circ (T_{\mu_l} - S_{\mu_l, x_0} )\circ  S_{\mu_{l-1}, x_0} \dots  \circ S_{\mu_{1}, x_0} \right)\varphi_{\delta}. 
\end{multline}
Fix $\beta_0=1<\beta_1<...<\beta_{2k-1}<\beta_{2(k+1)}=2$ with $\beta_{l+1} - \beta_l > 1/ (10d)$. Set 
$$
S_{\mu_l, x_0, 1} = \varphi_{\beta_l\delta} S_{\mu_l, x_0}, \mbox{ and }  \quad S_{\mu_l, x_0, 2} =  (1- \varphi_{\beta_l\delta}) S_{\mu_l, x_0}. 
$$
Then 
$$
\left( S_{\mu_{l-1}, x_0} \circ \dots  \circ S_{\mu_{1}, x_0} \right)\varphi_{\delta} = 
\left( (S_{\mu_{l-1}, x_0,  1} +  S_{\mu_{l-1}, x_0, 2}) \circ  \dots  \circ (S_{\mu_{1}, x_0, 1} + S_{\mu_{1}, x_0, 2} ) \right)\varphi_{\delta}. 
$$
Since $\varphi_{\beta_{l-1}\delta} =\varphi_{\beta_{l}\delta} \varphi_{\beta_{l-1}\delta}$, it follows from  \eqref{trace-sect1-a1} that 
\begin{multline}
\label{trace-sect1-a1'}
\varphi_{2\delta}\left(T_{\mu_{2(k+1)}}\circ \dots \circ T_{\mu_{1}}-  \cS_{t, x_0} \right)\varphi_{\delta} \\[6pt]
= \sum_{l=1}^{2(k+1)} \varphi_{2\delta}\left(T_{\mu_{2(k+1)}}\circ \dots  \circ T_{\mu_{l+1}}\right) \circ\left( (T_{\mu_l} - S_{\mu_l, x_0} )\varphi_{\beta_{l}\delta}\right)\circ  \left(S_{\mu_{l-1}, x_0,1} \circ  \dots  \circ S_{\mu_{1}, x_0,1} \varphi_{\delta}\right) + \\[6pt] \sum_{l=1}^{2(k+1)} \varphi_{2\delta}\left(T_{\mu_{2(k+1)}}\circ \dots \circ  T_{\mu_{l+1}} \right) \circ (T_{\mu_l} - S_{\mu_l, x_0} )\circ \left( S_{\mu_{l-1}, x_0} \circ  \dots  \circ S_{\mu_{1}, x_0}  - S_{\mu_{l-1}, x_0, 1} \circ  \dots  \circ S_{\mu_{1}, x_0, 1} \right)\varphi_{\delta}.
\end{multline}
Let $p_1=1<p_2<...<p_{2k} < d <p_{2(k+1)}< + \infty$  be such that 
$p_{l+1} < p_l d / (d- p_l) $ for $1 \le l \le 2k +1$.  Using the exponential decay property: for $\gamma > 1$, $r>0$, $y \in \mR^d$, and for $f$ with $\supp f \subset B_r$, it holds, for $t > r^{-3/2}$ 
\begin{equation}\label{expoS} \| S_{\mu_l} f \|_{L^\infty (\Omega  \setminus B_{\gamma r} (y) )} \le C_\gamma  e^{- c_\gamma r t } \| f\|_{L^q(B_r(y))},  
\end{equation}
one has for $l=2,..., 2(k+2)+1,$
\begin{equation}\label{trace-sect1-a2}
\| \left( S_{\mu_{l-1}, x_0} \dots  \circ S_{\mu_{1}, x_0} \right)\varphi_{\delta}  - 
\left( S_{\mu_{l-1}, x_0, 1} \dots  \circ S_{\mu_{1}, x_0, 1} \right)\varphi_{\delta} \|_{L^{1} \to L^{p_{l+1} }} \le C e^{- c \delta t}. 
\end{equation}
Combining  \eqref{trace-sect1-a1'}-\eqref{trace-sect1-a2}, and using \eqref{Z3}-\eqref{Z5} for $T_{\mu_l}$, and  \eqref{trace-S1}-\eqref{trace-S3} for $S_{\mu_l}$,  
it suffices to prove that 
\begin{equation}
\label{es-T-S-2-inf}
\|\varphi_{2 \delta}\left(T_{\mu_{2(k+1)}}-S_{\mu_{2(k+1)} ,x_0}\right)\|_{L^{p_{2(k+1)}}\to L^\infty} \leq C_{\delta_0} \Big(\Phi(\delta, x_0) + t^{-1/2} \delta^{-1} \Big) t^{-1+\frac{d}{2p_{2(k+1)}}},
\end{equation}
and  for  $l=1,2,...,2k+1$, 
\begin{align}
\label{es-T-S-1-2}
\|\left(T_{\mu_{l}}-S_{\mu_{l} ,x_0}\right)\varphi_{\beta_{l} \delta}\|_{L^{p_l}\to L^{p_{l+1}}}\leq  C_{\delta_0}\left(\Phi(\delta,x_0) +\delta^{-1}t^{-1/2}\right) t^{-1+\frac{d}{2}\left(\frac{1}{p_l}-\frac{1}{p_{l+1}}\right)}.
\end{align}

\medskip 
\noindent \underline{Step 1:} Proof of \eqref{es-T-S-2-inf}. We will prove the following stronger result, which will be used in the proof of \eqref{es-T-S-1-2}:  for $\lambda\in \cL(\theta, \Lambda_0) $ and $\sup_{n \in\mathbb{Z}}|\theta-n \pi|>\eps_0,$ $\beta\in [1,2]$; and for $1< p < d$,  $p \le q < \frac{pd}{d-p}$ or for $d>p$ and $q= + \infty$: 
\begin{equation}
\label{claimM}
\|\varphi_{\beta \delta} \left(T_{\lambda}- S_{\lambda ,x_0}\right)\|_{L^{p} \to L^q} \leq C_{\delta_0,\eps_0} \Big(\Phi(\delta, x_0) + |\lambda|^{-1/2} \delta^{-1} \Big)  |\lambda|^{-1+ \frac{d}{2} \Big(\frac{1}{p} - \frac{1}{q} \Big) }.
\end{equation}
\\
Denote 
$$
u = T_{\lambda} (f) \quad \quad \mbox{ and } \quad v= S_{\lambda, x_0} f. 
$$
Set 
$$
u_{j,  \delta} = \varphi_{\beta \delta} u_j \quad \mbox{ and } \quad v_{j,  \delta} = \varphi_{\beta \delta} v_j.
$$
Since, in $\Omega$,  
$$
\dive(A_j  \nabla u_{j}) - \lambda \Sigma_j u_{j}  = \Sigma_j f_j,
$$
and 
$$
\dive(A_j(x_0)  \nabla v_{j}) -\lambda \Sigma_j (x_0) v_{j}  = \Sigma_j(x_0) f_j,  
$$
we have, in $\Omega$,  
$$
\dive(A_j (x_0) \nabla u_{j, \delta}) - \lambda \Sigma_j (x_0) u_{j, \delta} = f_{j, \delta} \quad \mbox{ and } \quad  \dive(A_j(x_0) \nabla u_{j, \delta}) - \lambda \Sigma_j (x_0) u_{j, \delta} = g_{j, \delta}, 
$$
where 
$$
f_{j, \delta} = \tf_{j, \delta} + \dive \tF_{j, \delta} \quad \mbox{ and } \quad g_{j, \delta} = \tg_{j, \delta} + \dive \tG_{j, \delta}, 
$$
with 
$$
\tf_{j, \delta} = \varphi_{\beta\delta} \Sigma_j f_j +  A_j \nabla u_j \nabla  \varphi_{\beta\delta}  -\lambda \big(\Sigma_j  (x_0)-  \Sigma_j (x) \big) u_{j, \delta},
$$
$$
\tF_{j, \delta} = u_j A_j \nabla   \varphi_{\beta\delta}  +  \big( A_j (x_0) - A_j(x) \big) \nabla u_{j, \delta}, 
$$
$$
\tg_{j, \delta} = \varphi_{\beta\delta} \Sigma_j (x_0) f_j +  A_j (x_0 ) \nabla v_j \nabla \varphi_{\beta\delta},   \quad \mbox{ and } \quad \tG_{j, \delta} = v_j A_j  (x_0) \nabla \varphi_{\beta\delta}. 
$$
By \Cref{thm2} and \eqref{Z8}, we have  for $1< p < + \infty$, 
\begin{equation}\label{trace-p1}
 \| \nabla u\|_{L^p(\Omega)} + \| \nabla v\|_{L^p(\mathbb{R}^d)} + |\lambda|^{1/2}\left( \| u \|_{L^p(\Omega)}  +\| v \|_{L^p(\mathbb{R}^d)} \right)  \le C|\lambda|^{- \frac{1}{2}} \| f \|_{L^p(\Omega)}. 
\end{equation}
By   \Cref{lem1}, we obtain  for  $1 < p < + \infty$, 
\begin{multline}
|\lambda|^{1/2} \| \nabla (u_{j, \delta} - v_{j, \delta}) \|_{L^p(\mR^d)} +|\lambda|\| u_{j, \delta} - v_{j, \delta} \|_{L^p(\mR^d)} \\[6pt] 
\le C \Big( \|  \tf_{j, \delta} - \tg_{j, \delta}\|_{L^p(\mR^d)} + |\lambda|^{1/2} \| \tF_{j, \delta} - \tG_{j, \delta} \|_{L^p(\mR^d)} \Big). 
\end{multline}
Using \eqref{trace-p1},  we derive,  for $1 < p < + \infty$, that, with $u_{ \delta} =(u_{ 1,\delta},u_{ 2,\delta})$ and $v_{ \delta} =(v_{ 1,\delta},v_{ 2,\delta})$, 
\begin{equation*}
 \| \nabla (u_{ \delta} - v_{\delta}) \|_{L^p(\mR^d)} + |\lambda|^{1/2} \| u_{ \delta} - v_{\delta} \|_{L^p(\mR^d)} \le C_{\delta_0} \Big(\Phi(\delta, x_0) + |\lambda|^{-1/2} \delta^{-1} \Big) |\lambda|^{-1/2}  \| f\|_{L^p(\Omega)}.
\end{equation*}
By Gagliardo-Nirenberg's interpolation inequalities, one gets that for $1<p<d$ and $p\leq q\leq \frac{dp}{d-p}$ \begin{equation*}
\|u_{ \delta} - v_{\delta}\|_{L^{q}(\mathbb{R}^d)}\leq C_{\delta_0} \Big(\Phi(\delta, x_0) + |\lambda|^{-1/2} \delta^{-1} \Big)  |\lambda|^{-1+\frac{d}{2}\left(\frac{1}{p}-\frac{1}{q}\right)}\| f \|_{L^p(\Omega)},
\end{equation*}
and for $p>d$,  \begin{equation*}
\|u_{ \delta} - v_{\delta}\|_{L^{\infty}(\Omega)}\leq C_{\delta_0}\Big(\Phi(\delta, x_0) + |\lambda|^{-1/2} \delta^{-1} \Big)  |\lambda|^{-1+\frac{d}{2p}}\| f \|_{L^p(\Omega)},
\end{equation*}
and assertion \eqref{claimM} follows. 

\medskip 
\noindent \underline{Step 2: } Proof of \eqref{es-T-S-1-2}. By \eqref{adjopT}, 
$$
T_{\lambda}^\star-S_{\lambda,x_0}^\star=\begin{pmatrix}\Sigma_1 & 0\\0 & -\Sigma_2 \end{pmatrix}(T_{\overline{\lambda}}-S_{\overline{\lambda},x_0})\begin{pmatrix}1/\Sigma_1 & 0\\0 & -1/\Sigma_2 \end{pmatrix},$$
it follows that  by \eqref{claimM},  
 for $\overline{\lambda}\in \cL(\theta, \Lambda_0) $ with $\sup_{n \in\mathbb{Z}}|\theta- n \pi|>\eps_0,$ $\beta\in [1,2]$;  and for $1< p < d$,  $p \le q < \frac{pd}{d-p}$ or for $d>p$ and $q= + \infty$:
$$
\|\varphi_{\beta \delta}\left(T_{\lambda}^* - S_{\lambda,x_0}^*\right)\|_{L^{p} \to L^q}\leq C_{\delta_0,\eps_0} \Big(\Phi(\delta, x_0) + |\lambda|^{-1/2} \delta^{-1} \Big)  |\lambda|^{-1+ \frac{d}{2} \Big(\frac{1}{p} - \frac{1}{q} \Big) },
$$
which implies
\begin{align}
\|\left(T_{\lambda} - S_{\lambda,x_0}\right)\varphi_{\beta \delta}\|_{L^{\frac{q}{q-1}} \to L^{\frac{p}{p-1}}}\leq C_{\delta_0,\eps_0} \Big(\Phi(\delta, x_0) + |\lambda|^{-1/2} \delta^{-1} \Big)  |\lambda|^{-1+ \frac{d}{2} \Big(\frac{q-1}{q} - \frac{p-1}{p} \Big) }.
\end{align}
This gives \eqref{es-T-S-1-2}. The proof is complete. 
\end{proof}

As a consequence of  \Cref{pro-trace0} and  \Cref{pro-trace}, we obtain 

\begin{corollary} \label{cor-E} We have
\begin{align*}
\sum_{j} \frac{1}{|\tlambda_j|^{2(k+1)}- i t^{2(k+1)}}= {\bf \hat c} t^{\frac{d}{2}-2(k+1)} +o(t^{\frac{d}{2}-2(k+1)})~~\text{as}~~t\to\infty, 
\end{align*}
where  each characteristic value $\tlambda_j$ of $T_{\lambda_0}$ is repeated  a number of times equal to its multiplicity, and  ${\bf \hat c}$ is defined by \eqref{def-hc}. 
\end{corollary}
\begin{proof} By \Cref{pro-trace0,pro-trace}, and   \eqref{traceT^2} in  \Cref{lem-T1T2},  we have
\begin{align*}
\sum_{j} \frac{1}{\tlambda_j^{2(k+1)}- i t^{2(k+1)}}= {\bf \hat c} t^{\frac{d}{2}-2(k+1)} +o(t^{\frac{d}{2}-2(k+1)})~~\text{as}~~t\to\infty.
\end{align*}
For  $j_0 \in \mN$ large and for  $t\geq 2|\tlambda_{j_0}|$, we have 
\begin{multline}\label{coco1}
\left|\sum_{j=1}^{\infty}\frac{1}{\tlambda_j^{2(k+1)}-t^{2(k+1)}i}-\sum_{j=1}^{\infty}\frac{1}{|\tlambda_j|^{2(k+1)}-t^{2(k+1)}i}\right| \\[6pt]
\leq \sum_{j=1}^{j_0}\frac{1}{|\tlambda_j^{2(k+1)}-t^{2(k+1)}i|}+\frac{1}{\big| |\tlambda_j|^{2(k+1)}-t^{2(k+1)}i \big|} \\[6pt]
+\sum_{j=j_0+1}^{\infty}\frac{C |\tlambda_j|^{2k +1} |\Im \tlambda_j|}{|\tlambda_j^{2(k+1)}-t^{2(k+1)}i| \, \big||\tlambda_j|^{2(k+1)}-t^{2(k+1)}i \big|}\\[6pt]
 \leq 2j_0 t^{-2(k+1)}+\left( \sup_{j \ge j_0} \frac{|\Im \tlambda_j|}{|\tlambda_j|} \right) \sum_{j=j_0+1}^{\infty} \frac{1}{(|\tlambda_j|+t)^{2(k+1)}}. 
\end{multline}
By  \cite[Theorem 12.14]{Agmon} and \Cref{thm2}, we have
\begin{equation}\label{coco2}
\sum_{j \ge j_0}^\infty (|\tlambda_j| + |t|)^{-2(k+1)}  \le C \vvvert T_{-it}^{k+1} \vvvert^2,
\end{equation}
and  by \Cref{pro-HS}, we obtain 
\begin{equation}\label{coco3}
\vvvert T_{- i t} \vvvert^2 \le C t^{\frac{d}{2} - 2(k+1)}. 
\end{equation}
Since, by \Cref{thm2},  
$$
\left( \sup_{j \ge j_0} \frac{|\Im \tlambda_j|}{|\tlambda_j|} \right) \to 0 \mbox{ as } j_0 \to + \infty, 
$$
it follows from \eqref{coco1}, \eqref{coco2}, and \eqref{coco3} that 
\begin{equation*}
\left|\sum_{j=1}^{\infty}\frac{1}{\tlambda_j^{2(k+1)}-t^{2(k+1)}i}-\sum_{j=1}^{\infty}\frac{1}{|\tlambda_j|^{2(k+1)}-t^{2(k+1)}i}\right|=  \circ(1) t^{\frac{d}{2}-2(k+1)}~~\text{ as}~t\to\infty.
\end{equation*}
The proof is complete. 
\end{proof}

\subsection{Proof of \Cref{thm1}} 

Before giving the proof of \Cref{thm1}, we  recall  a Tauberian theorem of Hardy and Littlewood, see e.g.  \cite[Theorem 2a]{Widder41} or \cite[Theorem 14.5]{Agmon}. 
\begin{lemma}\label{tauberian-thm}
Let $\sigma(s)$ be a non-decreasing function for $s>0$, let $a\in (0,1)$ and  $P\geq 0$. Then,  as $t\to \infty$, 
$$
\int_{0}^{\infty}\frac{d\sigma(s)}{s+t}=Pt^{a-1}+\circ (t^{a-1}), 
$$
if and only, as $s\to\infty$, 
$$
\sigma(s)=\frac{P}{a\int_{0}^{\infty}t^{a-1}(1+t)^{-1}dt}s^a+\circ (s^a).$$
\end{lemma}

We are ready to give 
\begin{proof}[Proof of \Cref{thm1}]
We have, by  \Cref{cor-E}, 
\begin{align*}
\sum_{j=1}^{\infty}\frac{1}{|\tlambda_j|^{2(k+1)}-t^{2(k+1)}i}= {\bf \hat c} t^{\frac{d}{2}-2(k+1)} +o(t^{\frac{d}{2}-2(k+1)})~~\text{as}~~t\to\infty, 
\end{align*}
where  ${\bf \hat c}$ is given by \eqref{def-hc}:  
\begin{align} 
{\bf \hat c} =\frac{1}{(2\pi)^d}\int_\Omega  
\sum_{j=1,2}\int_{\mathbb{R}^d}\frac{1}{(\Sigma_j(x)^{-1}\langle  A_j(x) \xi, \xi \rangle)^{2(k+1)}- i}d\xi dx.
\end{align}
Considering the imaginary part yields, 
\begin{align*}
\sum_{j=1}^{\infty}\frac{t^{2(k+1)}}{|\tlambda_j|^{4(k+1)}+t^{4(k+1)}}= {\bf \hat c}_1  t^{\frac{d}{2}-2(k+1)} +o(t^{\frac{d}{2}-2(k+1)})~~\text{as}~~t\to+\infty, 
\end{align*}
where 
$$
{\bf \hat c}_1 = \Im ({\bf \hat c}) =  \frac{1}{(2\pi)^d}\int_\Omega  
\sum_{j=1,2}\int_{\mathbb{R}^d}\frac{1}{(\Sigma_j(x)^{-1}\langle  A_j(x) \xi, \xi \rangle)^{4(k+1)}+1}d\xi dx. 
$$
This implies, by replacing $t^{4(k+1)}$ by $t$, 
\begin{align*}
\sum_{j=1}^{\infty}\frac{1}{|\tlambda_j|^{4(k+1)}+t}= {\bf \hat c}_1  t^{\frac{d}{8(k+1)}-1} +o(t^{\frac{d}{8(k+1)}-1})~~\text{as}~~t\to+\infty.
\end{align*}
Since $\tlambda_j=\lambda_j-\lambda_0$, one obtains 
\begin{align*}
\sum_{j=1}^{\infty}\frac{1}{|\lambda_j|^{4(k+1)}+t}= {\bf \hat c}_1  t^{\frac{d}{8(k+1)}-1} +o(t^{\frac{d}{8(k+1)}-1})~~\text{as}~~t\to+\infty.
\end{align*}
We can write this identity under the form 
\begin{align*}
\int_{0}^{\infty}\frac{dN(s^{\frac{1}{4(k+1)}})}{s+t} = {\bf \hat c}_1  t^{\frac{d}{8(k+1)}-1} +o(t^{\frac{d}{8(k+1)}-1})~~\text{as}~~t\to\infty.
\end{align*}
By Lemma \ref{tauberian-thm}, one has
\begin{align}
N(t)=\mathbf{c}  t^{\frac{d}{2}} +\circ(t^{\frac{d}{2}}),
\end{align}
where 
$$
\mathbf{c}= \frac{{\bf \hat c}_1}{\frac{d}{8(k+1)}\int_{0}^{\infty}t^{\frac{d}{8(k+1)}-1}(1+t)^{-1}dt}.$$
We have, by Fubini's theorem, 
\begin{align*}
\int_{\mathbb{R}^d}\frac{1}{(\Sigma_j(x)^{-1}\langle  A_j(x) \xi, \xi \rangle)^{4(k+1)}+1}d\xi=
\int_{0}^{\infty} \Big|\Big\{\xi:(\Sigma_j(x)^{-1}\langle  A_j(x) \xi, \xi \rangle)^{4(k+1)}<t \Big\} \Big|  \frac{dt}{(t+1)^2}.
\end{align*}
Since 
$$\Big| \Big\{\xi:(\Sigma_j(x)^{-1}\langle  A_j(x) \xi, \xi \rangle)^{4(k+1)}<t \Big\} \Big|=t^{\frac{d}{8(k+1)}} \Big| \Big\{\xi:\langle  A_j(x) \xi, \xi \rangle< \Sigma_j(x) \Big\} \Big|, $$
it follows that 
\begin{align*}
\int_{\mathbb{R}^d}\frac{1}{(\Sigma_j(x)^{-1}\langle  A_j(x) \xi, \xi \rangle)^{4(k+1)}+1}d\xi &= \Big| \Big\{\xi:\langle  A_j(x) \xi, \xi \rangle< \Sigma_j(x) \Big\} \Big|
\int_{0}^{\infty} \frac{t^{\frac{d}{8(k+1)}} dt}{(t+1)^2}
\\&= \Big| \Big\{\xi:\langle  A_j(x) \xi, \xi \rangle< \Sigma_j(x) \Big\} \Big|
\frac{d}{8(k+1)}\int_{0}^{\infty}t^{\frac{d}{8(k+1)}-1}(1+t)^{-1}dt. 
\end{align*}
Here in the last identity, an integration by parts is used. We therefore have 
$$
\mathbf{c}=\frac{1}{(2\pi)^d}  \sum_{j=1,2}  \int_\Omega \Big|\Big\{\xi:\langle  A_j(x) \xi, \xi \rangle< \Sigma_j(x)\Big\} \Big| dx.
$$
The proof is complete. 
\end{proof}

\section{Completeness of generalized eigenfunctions of the transmission problem - Proof of \Cref{thm1-C}} \label{sect-C}

Fix $\eps_0>0$. For $z \in \cL (\theta, 1)$ with $\inf_{n \in \mZ}|\theta - n \pi| > \eps_0$ and $|z|$ large enough, let $\tau_1, \cdots, \tau_{k+1}$ with $k = k_d = [d/2]$ be the $k+1$ distinct roots in $\mC$ of the equation $x^{k+1} = z$. Set 
$$
\eta_l = \lambda_0 + \tau_l \mbox{ for } 1 \le l \le k+1.  
$$
As in the proof of \eqref{TTT}, one has 
$$
T_{\eta_{k+1}} \circ \dots \circ T_{\eta_1} = T_{\lambda_0}^{k+1} \Big(I - z T_{\lambda_0}^{k+1} \Big)^{-1}. 
$$
It follows that 
$$
T_{\eta_{k+1}} \circ \dots \circ T_{\eta_1} = \Big(T_{\lambda_0}^{k+1}\Big)_{z}. 
$$
Since $T_{\lambda_0}^{k+1}$ is a Hilbert-Schmidt operator, it follows from  \cite[Theorem 16.4]{Agmon} that: 

1) the space spanned by the general eigenfunctions of 
$T_{\lambda_0}^{k+1}$ is equal to $\overline{{\bf R}(T_{\lambda_0}^{k+1})}$,  the closure of the range of $T_{\lambda_0}^{k+1}$ with respect to the $L^2$-topology. 

\noindent On the other hand, we have 

2) the range ${\bf R}(T_{\lambda_0}^{k+1})$ of  $T_{\lambda_0}^{k+1}$ is dense in $[L^2(\Omega)]^2$, since ${\bf R}(T_{\lambda_0})$ is dense in $[L^2(\Omega)]^2$ and $T_{\lambda_0}$ is continuous,

3) the space spanned by the general eigenfunctions of 
$T_{\lambda_0}^{k+1}$ is equal to the space spanned by the general eigenfunctions of  $T_{\lambda_0}$. \\

The conclusion now follows from 1), 2), and 3). \qed

\providecommand{\bysame}{\leavevmode\hbox to3em{\hrulefill}\thinspace}
\providecommand{\MR}{\relax\ifhmode\unskip\space\fi MR }
\providecommand{\MRhref}[2]{%
  \href{http://www.ams.org/mathscinet-getitem?mr=#1}{#2}
}
\providecommand{\href}[2]{#2}

\end{document}